\documentclass[11pt]{amsart}
\usepackage{xypic}
\usepackage{amsmath,amssymb,amscd,amsthm,graphicx,enumerate}

 \usepackage{float}
\usepackage{hyperref}

\usepackage{xcolor}

\usepackage{mathtools}
\usepackage{enumitem}
\usepackage[normalem]{ulem}
\hypersetup{breaklinks=true}
\hypersetup{
    colorlinks=true,
    linkcolor=red,
    filecolor=magenta,      
    urlcolor=cyan,
    citecolor=green,
}

\numberwithin{equation}{section}
\usepackage[
  left=1.3in,
  right=1.3in,
  top=1.05in,
  bottom=1.05in
]{geometry}

\newcommand{\cJ}{\mathcal{J}}

\newcommand{\dist}{{\mathrm{dist}}}
 
\newcommand{\eps}{\varepsilon}
\newcommand{\be}{\begin{equation}}
\newcommand{\bee}{\begin{equation*}}
\newcommand{\bea}{\begin{equation}\begin{aligned}}
\newcommand{\eea}{\end{aligned}\end{equation}}
\newcommand{\ee}{\end{equation}}
\newcommand{\eee}{\end{equation*}}
\newcommand{\bsp}{\begin{split}}
\newcommand{\esp}{\end{split}}

\newcommand{\R}{{\mathbb R}}
\newcommand{\supp}{\mathop{\mathrm {supp}}}

\newtheorem{theorem}{Theorem}[section]
\newtheorem{proposition}[theorem]{Proposition}
\newtheorem{lemma}[theorem]{Lemma}

\theoremstyle{remark}
\newtheorem{remark}[theorem]{Remark}

\theoremstyle{remark}

\theoremstyle{definition}
\newtheorem{definition}[theorem]{Definition}

\renewcommand{\cH}{\mathcal{H}}

\usepackage{fancyhdr}
\usepackage{fancyhdr}

\usepackage{fancyhdr}

\title[Sharp Dynamical Isoperimetric Principle For Mean Curvature Flow]{Sharp dynamical isoperimetric principle for mean curvature flow}

\dedicatory{\normalfont\rmfamily IN MEMORY OF PROFESSOR TOM ILMANEN} 

\author[Beomjun Choi, Wenkui Du, Seung Chul Park, Junseo Youn]{Beomjun Choi}
\author[Beomjun Choi, Wenkui Du, Seung Chul Park, Junseo Youn]{Wenkui Du}
\author[Beomjun Choi, Wenkui Du, Seung Chul Park, Junseo Youn]{Seung Chul Park}
\author[Beomjun Choi, Wenkui Du, Seung Chul Park, Junseo Youn]{Junseo Youn}

\begin{document}

\begin{abstract}
In this paper, we establish a sharp dynamical isoperimetric principle for weak mean curvature flow. We prove that, among weak mean curvature flows with fixed area of initial smooth closed hypersurfaces, the standard smoothly shrinking spherical mean curvature flow uniquely maximizes the extinction time. More precisely,  for the level set flow $K_t$ starting from the boundary of a smooth bounded domain $\Omega\subset\mathbb R^{n+1}$, as well as for the integral Brakke flow $\{\mu_t\}$ with initial Radon measure $\mu_0=\mathcal H^n\llcorner\partial\Omega$, we establish  the corresponding optimal extinction time estimates
\begin{equation*}
        T_{\rm ext}(\Omega),\, T^B_{\rm ext}(\Omega)
\leq
\frac{1}{2n}
\left(\frac{P( \Omega)}{|{\mathbb{S}^n}|}\right)^{\frac{2}{n}},
\end{equation*}
where $P(\Omega)$ is the perimeter of $\Omega$ representing the area of $\partial \Omega$, and the equality holds if and only if  $\Omega$ is a round ball and the flow is the standard multiplicity-one   smoothly self-shrinking round sphere. In particular, we obtain the sharp $L^p$-estimates for the arrival time function of a smooth bounded mean convex domain. In addition, we also establish the  sharp isoperimetric inequality for the parabolic measure of  the space-time track filling $X$ of outward minimizing level set  flow starting from  the boundary of a smooth bounded domain $\Omega\subset \mathbb R^{n+1}$:
\begin{equation*}
\mathcal H_{\mathrm{par}}^{n+2}(X)
\leq
\frac{\pi}
{2n(n+2)^2|{\mathbb{S}^n}|^{{\frac{2}{n}}}}
P(\Omega)^{\frac{n+2}{n}},
\end{equation*}
where the equality holds if and only if  $\Omega$ is a round ball and the flow is standard  smoothly shrinking round sphere. 
\end{abstract}
\maketitle

\tableofcontents

\section{Introduction}\label{intro}
Mean curvature flow is one of the fundamental geometric evolution equations, describing the deformation of a closed hypersurface in the direction of its mean curvature vector. It arises naturally as the negative gradient flow of the area functional
and has played a central role in geometric analysis, both as a mechanism for
deforming submanifolds toward canonical configurations and as a model for
the formation of geometric singularities. Mean curvature flow and  its variants have been used to study geometric inequalities. For instance,  Topping \cite{Topping1998}  applied curve shortening flow to obtain an optimal
isoperimetric inequality on surfaces. Schulze \cite{schulze2008nonlinear,schulze2020optimal} employed mean curvature flow and its nonlinear variants to establish sharp isoperimetric inequalities in several settings, including optimal isoperimetric inequalities for surfaces in Cartan--Hadamard manifolds. In another direction, Guan and Li \cite{GuanLi2009} employed inverse curvature flows to establish Alexandrov--Fenchel type inequalities, including a Minkowski-type inequality for star-shaped mean convex hypersurfaces. Huisken also used the weak inverse mean curvature flow approach developed in his joint work with Ilmanen \cite{HuiskenIlmanen2001} to obtain the Minkowski-type inequality for more general outward minimizing hypersurfaces. Moreover, 
Huisken and Ilmanen \cite{HuiskenIlmanen2001} applied the weak inverse mean curvature flow and used its monotonicity properties to prove the Riemannian Penrose inequality for asymptotically flat three-manifolds.
Bray \cite{bray2001proof} and Bi-Zhu \cite{BZ26Penrose} used conformal flow to prove general Riemannian Penrose inequality.

On the other hand,  the connection between mean curvature flow and isoperimetric inequalities has also been studied from the viewpoint of extinction time and filling volume estimates.  The classical isoperimetric inequality compares the area and enclosed volume of a hypersurface at a fixed time, whereas an extinction time estimate relates the geometry of the initial hypersurface to the lifespan of the mean curvature flow evolution it generates. For example, the lifespan of the curve shortening flow of simple closed curve in $\mathbb{R}^2$ is completely determined by the enclosed area $A$ of initial simple closed curve, so, by the planar isoperimetric inequality, among  curve shortening flows with fixed length $L$ of initial closed curves, the shrinking circle uniquely maximizes the extinction time $\frac{L^2}{8\pi^2}$. {Moreover, from the viewpoint of filling-volume estimates, Hershkovits \cite{hershkovits2017isoperimetric} explored parabolic version isoperimetric inequality for the mean curvature flow: the parabolic volume of space-time track of flow is estimated in terms of the area of initial submanifold viewed as the boundary of space-time track filling.} One of the crucial ingredients of Hershkovits's proof is an  upper bound estimate of the extinction time \cite[Lemma 1.12]{hershkovits2017isoperimetric} from Huisken's monotonicity formula. Hence, from the viewpoint of geometric inequalities,  extinction time estimates for mean curvature flow can be regarded as a parabolic analogue of isoperimetric
inequality, and the extinction time estimates and parabolic version isoperimetric inequality provide  dynamical counterparts of the classical isoperimetric principle.

Now, we recall the previously established results about the extinction time estimates besides the  related  extinction time estimates results of Hershkovits \cite[Theorem A, Lemma 1.12]{hershkovits2017isoperimetric} for enhanced motions in arbitrary codimension of mean curvature flow in Euclidean space. A sharp lower bound for extinction time of mean curvature flow of hypersurfaces in Euclidean space has been obtained in the sense of level set flow \cite{GigaYamauchi1993}  and the  canonical multi-phase Brakke flow constructed in \cite{StuvardTonegawa2024}, namely,
\begin{equation}
    T_{\textrm{ext}}\geq 2\left(\frac{|\Omega|}{P(\Omega)}\right)^2,
\end{equation}
where $P(\Omega)$ is the perimeter of $\Omega$. 

Then, for the upper bound estimates for extinction time of mean curvature flow of hypersurfaces in Euclidean space, it is natural to ask if the extinction time of  mean curvature flow  can be bounded above in terms of the geometric information such as the area of the initial hypersurface\footnote{Colding-Minicozzi 
\cite{ColdingMinicozzi2008} also obtained sharp bounds
for the extinction time of mean curvature flow in terms of the min-max
width. 
For  convex closed mean curvature flow, they proved $T_{\mathrm{ext}}\leq \frac{W(0)}{4\pi}=\frac{1}{2}\left(\frac{W_1(0)}{|\mathbb{S}^1|}\right)^2$ and $T_{\mathrm{ext}}\leq \frac{1}{2\times 2}\frac{W_2(0)}{|\mathbb{S}^2|}$. 
They also studied the extinction time of Ricci flow in terms of the min-max
width  (see \cite{colding2005estimates, colding2008width}).}. 
%However, depending on the initial hypersurface, smooth mean curvature flow can attain positive volume at their first singular time.
Evans and Spruck \cite{evans1992motion} obtained the following nonsharp upper bound on extinction time  for level set flow 
\begin{equation}\label{sharpextinctionintro}
    T_{\textrm{ext}} \leq C_n\left(\frac{P(\Omega)}{|\mathbb{S}^n|}\right)^{{\frac{2}{n}}}.
\end{equation}
Under the assumptions of mean convexity and smoothness of every time slice of the level set flow prior to extinction, Brandolini, Cicalese, Nitsch, and Trombetti \cite{brandolini2007sharp} used a symmetrization argument for  the arrival time function to show that the sharp extinction time estimate \eqref{sharpextinctionintro} with
\begin{equation}
    C_n=\frac{1}{2n}.
\end{equation}
For  general mean curvature flow of closed hypersurfaces,  singularities may develop  before extinction time, and the singular set can have a highly complicated structure. Consequently, the classical evolution equations used in the smooth setting are no longer directly available across singular times, and the argument of \cite{brandolini2007sharp} does not immediately extend to such flows. This naturally raises the question of whether the above same sharp extinction time estimate continues to hold in the presence of general complicated singularities, in particular for weak notions of mean curvature flow such as level set flow or Brakke flow. 

\subsection{Main results} 
Now, we state our main results on optimal extinction time estimates and parabolic isoperimetric estimates for weak mean curvature flow such as level set flow (the unique maximal closed set-theoretic weak mean curvature flow satisfying the avoidance principle with respect to smooth mean curvature flows) and Brakke flow (a measure-theoretic weak mean curvature flow satisfying Brakke’s inequality, see \cite{Ilm94, SchulzeBrakkeFlows} or Definition \ref{def-Brakke flow}).  

\begin{theorem}[Optimal extinction time estimates for weak mean curvature flow]\label{Optimal extinction time estimates for weak mean curvature flow}
Let $\Omega\subset\R^{n+1}$ be a bounded domain with smooth boundary.
Let $\{K_t\}_{t\geq 0}$ be the level set flow starting from $\Omega$ and  $\{\mu_t\}_{t\geq 0}$ be an integral Brakke flow starting from the associated Radon measure $\mu_0=\mathcal{H}^n\llcorner\partial\Omega$. Then, for extinction times $T_{\mathrm{ext}}(\Omega)
:=
\inf\{t\geq0:K_t=\varnothing\}$ and $T^{B}_{\mathrm{ext}}(\Omega)
:=
\inf\{t\geq0:\mu_t\equiv0\},$
we have
\begin{equation}\label{C1}
           T_{\rm ext}
(\Omega)\leq
\frac{1}{2n}
\left(\frac{P( \Omega)}{|{\mathbb{S}^n}|}\right)^{\frac{2}{n}}
\end{equation}
and
\begin{equation}\label{C2}
           T^B_{\rm ext}
(\Omega)\leq
\frac{1}{2n}
\left(\frac{P( \Omega)}{|{\mathbb{S}^n}|}\right)^{\frac{2}{n}}
\end{equation}
respectively. In addition,
the equality in \eqref{C1} holds  if and only if $\Omega$ is a round ball, and 
the equality in \eqref{C2} holds if and only if $\Omega$ is a round ball and the Brakke flow $\{\mu_t\}$ is the multiplicity-one standard  smoothly shrinking round sphere Brakke flow. 
\end{theorem}
\begin{remark}
Indeed,  Theorem \ref{Optimal extinction time estimates for weak mean curvature flow} holds when $\Omega$ is a bounded Lipschitz domain. The readers can refer to Remark \ref{remarkonsmooth issue} after the proof of our main results for more explanations. 
\end{remark}
\begin{remark}
    In particular, above Theorem \ref{Optimal extinction time estimates for weak mean curvature flow} admits a dynamical isoperimetric principle interpretation: among weak mean curvature flows starting from smooth closed hypersurfaces of fixed area, the standard smoothly shrinking spherical mean curvature flow maximizes the lifespan. This can be easily understood in the planar curve shortening flow ($n=1$ case) since planar curve shortening flow of simple closed curve remains smooth until it extincts and its lifespan is completely determined by the enclosed area  of initial simple closed curve. For the situation when $\partial \Omega$ has multiple components, one can still use the differential formula of length and  Fenchel's total absolute curvature inequality to get a direct proof of Theorem \ref{Optimal extinction time estimates for weak mean curvature flow} when $n=1$.  
    On the other hand, Theorem \ref{Optimal extinction time estimates for weak mean curvature flow}  can be interpreted as optimal  extinction time estimates in terms of the first $n$-dimensional min-max width of smooth $n$-dimensional closed hypersurface $\partial \Omega$ (see problem (5) in Remark \ref{problems}  for more details). 
\end{remark}
\begin{remark}
In addition, we also have the following version of the optimal extinction
time estimate (see also Proposition \ref{thm:LSF-Brakke extinction optimal}) used in the proof of Theorem \ref{Optimal extinction time estimates for weak mean curvature flow}: 
\begin{equation}\label{eq:general-extinction}
T_{\mathrm{ext}}(\Omega)
\leq
\frac{1}{2n}
\left(
\frac{P(\Omega^*)}{|{\mathbb{S}^n}|}
\right)^{\frac{2}{n}}
\end{equation}
and
\begin{equation}\label{eq:general-brakke-extinction}
T^{B}_{\mathrm{ext}}(\Omega)
\leq
\frac{1}{2n}
\left(
\frac{P(\Omega^*)}{|{\mathbb{S}^n}|}
\right)^{{\frac{2}{n}}}
\end{equation}
hold, where $P(\Omega^*)$ is the perimeter of the outward minimizing hull $\Omega^*$ of $\Omega$.
In addition, the equality in \eqref{eq:general-extinction} holds,
if and only if $\Omega^*$ is a round ball. In particular, if $\Omega$ has only one smooth boundary component, then equality in \eqref{eq:general-extinction}
if and only if $\Omega=\Omega^*$ is a round ball. Moreover, if $\Omega$ has only one smooth boundary component, then equality in \eqref{eq:general-brakke-extinction} holds, if and only if $\Omega=\Omega^*$ is a round ball and the Brakke flow $\{\mu_t\}$ is the multiplicity-one standard  smoothly shrinking round sphere Brakke flow.
\end{remark}
For a bounded smooth mean convex domain, the extinction time  can be interpreted as  $L^{\infty}$-norm of arrival time function. In particular, we have the following result.
\begin{theorem}[$L^p$-estimates of arrival time function]\label{lpestimates arrival time}
   Let $\Omega\subset \mathbb{R}^{n+1}$ be  a bounded smooth mean convex domain and $u(x):=\sup\{t\geq 0:x\in K_t\}$ be the arrival time function such that $u=0$ on $\partial \Omega$. Then, we have
   \begin{equation}\label{eqlinfty}
\|u\|_{L^\infty(\Omega)}
\le
\frac{1}{2n}
\left(
\frac{P(\Omega)}{|\mathbb S^n|}
\right)^{\frac{2}{n}}
    \end{equation}
and
\begin{equation}\label{eqlp}
\|u\|_{L^p(\Omega)}
\leq
\frac{1}{2n}
\left[
\frac{1}{2}
\mathrm{B}\!\left(\frac{n+1}{2},\,p+1\right)
\right]^{\frac{1}{p}}
|\mathbb{S}^n|^{-\frac{2p+1}{np}}
P(\Omega)^{\frac{n+1+2p}{np}},\qquad p>0, 
\end{equation}
where 
\begin{equation}\label{eqn:beta}
\mathrm{B}(a,b):=\int_0^1 s^{a-1}(1-s)^{b-1}\,ds,
\qquad a,b>0
\end{equation}
is the standard {Beta} function. In particular, the equality in \eqref{eqlinfty}, \eqref{eqlp} holds if and only if $\Omega$ is a  round ball.
\end{theorem}

By the fact that mean curvature flow is negative gradient flow of area functional and \cite{magni2015variational}, it selects a particularly canonical  and efficient space-time filling  of initial closed smooth hypersurface in $\mathbb{R}^{n+1}$ by sweeping out a track in space–time $\mathbb{R}^{1, n+1}$, and the mean curvature flow space-time track filling has the initial closed smooth hypersurface as its boundary. In the mean curvature flow theory, one should use parabolic Hausdorff measure defined below to measure the size of the mean curvature flow space-time track filling. 
\begin{definition}[cf. \cite{White1997Stratification, hershkovits2017isoperimetric}]
The {parabolic metric} on $\mathbb{R}^{1,n+1}$ is defined to be
\begin{equation}
    d_{\mathrm{par}}\bigl((t,x),(s,y)\bigr)
    =
    \max\left\{
        \sqrt{|t-s|},
        |x-y|
    \right\}.
\end{equation}
The {$(n+2)$-dimensional} {parabolic Hausdorff measure} $\mathcal{H}^{n+2}_{\mathrm{par}}$ induced by $d_{\mathrm{par}}$ is given by
\begin{equation}
\mathcal H_{\mathrm{par}}^{n+2}(A)
:=
\lim_{\delta\downarrow0}
\inf\left\{
\frac{\omega_{n+2}}{2^{n+2}}
\sum_i
\big(\operatorname{diam}_{\mathrm{par}} U_i\big)^{n+2}
:
A\subset\bigcup_i U_i,\ 
\operatorname{diam}_{\mathrm{par}} U_i<\delta
\right\},
\end{equation}
where $\omega_{n+2}$ denotes the volume of the Euclidean unit ball in
$\mathbb R^{n+2}$.
\end{definition}

Motivated by the classical isoperimetric inequality, it is natural to consider a parabolic analogue of the isoperimetric problem as discussed in \cite[Theorem C]{hershkovits2017isoperimetric}:
{Given the area of the initial closed smooth hypersurface 
$M_0=\partial \Omega\subset \mathbb{R}^{n+1}$ as boundary of level set flow space-time track filling, can one establish an optimal, 
parabolic-scaling-invariant isoperimetric type bound on the parabolic 
Hausdorff measure of the entire level set flow space-time track 
filling in terms of 
$P(\Omega)^{\frac{n+2}{n}}$?}

By the parabolic co-area formula in \cite[Theorem D]{hershkovits2017isoperimetric} and \cite[Example 3.5]{hershkovits2017isoperimetric}, the parabolic Hausdorff measure of entire
mean curvature flow space–time track filling under reasonable condition can be written as constant multiple of time integral of perimeter of level set flow time slices $K_t$ over $[0, T_{\textrm{ext}}]$, where $K_0=\Omega$ has smooth boundary.
Hence, the above problem  under reasonable condition can be interpreted as
establishing an optimal estimates of the following parabolic scaling invariant level set flow time slice area integrated form inequality:
\begin{equation}
\int_0^{T_{\mathrm{ext}}}P(K_t)dt
\leq
C_n
P(\Omega)^{\frac{n+2}{n}},
\end{equation}
where $0<C_n<\infty$ is the sharp constant to be determined.

Concerning this question, in our following result, we obtain the optimal parabolic isoperimetric estimates for the space-time
track filling of outward minimizing mean curvature flow.
Since outward minimizing property implies mean convexity and therefore
non-fattening, there exists a canonical unit-density matching motion
corresponding to the level set flow \cite[Theorem 11.4]{Ilm94}.

\begin{theorem}[Optimal parabolic isoperimetric estimates for outward-minimizing level set flow]\label{sharp parabolic estimates}
Let $\Omega\subset\mathbb{R}^{n+1}$ be a bounded
domain with smooth outward-minimizing boundary, and let $\{K_t\}_{t\geq 0}$ be an outward minimizing level-set flow starting from
$K_0={\Omega}$.
Then, there is a unique canonical unit-density matching motion $(T,\{\mu_t\}_{t\geq0})$ 
associated with this level-set flow, such that  for  the sliced current $T_t$ and its associated mass Radon measure $\mu_{T_t}$ satisfy
\begin{equation}
T_t=\partial[K_t],
\qquad
\mu_{T_t}
=\mu_t=\mathcal H^n\llcorner\partial^*K_t\quad\,\,\textrm{for}\,\, \mathrm{a.e}.\,\,t.
\end{equation}
In particular, for the
rectifiable set of level set flow space-time track $X\subset\R^{1, n+1}$ associated to the above current  $T$, we have the following optimal parabolic isoperimetric estimate: 
\begin{equation}
\label{eq:sharp-parabolic-matching}
\mathcal H_{\mathrm{par}}^{n+2}(X)
\leq
\frac{\pi}
{2n(n+2)^2|{\mathbb{S}^n}|^{{\frac{2}{n}}}}
P(\Omega)^{\frac{n+2}{n}} .
\end{equation}
The equality holds if and only if $\Omega$ is a round ball and the flow is standard  smoothly shrinking round sphere.
\end{theorem}
\begin{remark}
Theorem \ref{sharp parabolic estimates} sharpens Hershkovits's parabolic isoperimetric inequality \cite[Theorem C]{hershkovits2017isoperimetric} for level set flows starting from bounded domains with smooth outward-minimizing boundary, yielding the optimal constant and characterizing the shrinking round sphere as the unique equality case.  Moreover, for level set flow starting from smooth bounded domain with outward minimizing boundary, using co-area formula and \cite[Theorem D, Example 3.5]{hershkovits2017isoperimetric}, the above
 parabolic isoperimetric estimates can be interpreted as optimal $L^1$ estimates on the gradient of arrival time function $u$  of outward minimizing level set flow
 \begin{equation}
      \int_{\Omega} |\nabla u|\,dx
\leq
\frac{P(\Omega)^{\frac{n+2}{n}}}
{n(n+2)|\mathbb{S}^n|^{{\frac{2}{n}}}}.
  \end{equation}
In addition, for the further discussion about the improvement, stability and optimal parabolic isoperimetric estimates in the situation of two dimensional surfaces in Cartan-Hadamard manifolds, the readers can refer to the end of Section \ref{proofofmainresults}. 
\end{remark}

\subsection{Proof strategy}\label{strategy} 
Now, we briefly discuss key ideas, difficulties and crucial ingredients. 
We first consider the toy model case, in which the  mean curvature flow   of closed hypersurfaces $\{M_t\}_{0\leq t\leq T_{\textrm{ext}}}$ is outward minimizing and smooth until its extinction time. By Cauchy–Schwarz inequality and Minkowski inequality for outward minimizing closed hypersurface noticed by Huisken (see \cite{HuiskenIlmanen2001} and see also \cite[Theorem 2(b)]{FreireSchwartz2014} for a published proof), we obtain
\be\label{toymtode} \frac{d}{dt}|M_t| = -\int_{M_t}H^2 \leq -|M_t|^{-1}\left(\int_{M_t}H\right)^2\leq -n^2|\mathbb{S}^n|^{\frac2n}|M_t|^{\frac{n-2}n}. \ee
Integrating \eqref{toymtode} until extinction time, one has the optimal estimate
\begin{equation}\label{toytestimate}
    T_{\textrm{ext}} \leq {\frac{1}{2n}}\left(\frac{|M_0|}{|\mathbb{S}^n|}\right)^{{\frac{2}{n}}}
\end{equation}
To analyze the equality case of \eqref{toytestimate}, we let $\Phi(t)=|M_t|^{\frac{2}{n}}+2n|\mathbb{S}^n|^{\frac2n}t$. By \eqref{toymtode}, $\Phi(t)$ is nonincreasing. By \eqref{toytestimate} and $\Phi(0)=|M_0|^{\frac{2}{n}}$, we have
\begin{equation}
    |M_t|^{\frac{2}{n}}=|M_0|^{\frac{2}{n}}-2n|\mathbb{S}^n|^{\frac2n}t
\end{equation}
Hence
\begin{equation}
    \frac{d}{dt}|M_t|^{\frac{2}{n}}=-2n|\mathbb{S}^n|^{\frac2n}
\end{equation}
and
\begin{equation}
\int_{M_t} H^2\,d\mu_t
=
n^2 |\mathbb{S}^n|^{{\frac{2}{n}}}
|M_t|^{\frac{n-2}{n}}.
\end{equation}
In particular, at $t=0$, for the mean curvature $H_0$ of $M_0$
\begin{equation}\label{H02integral=}
    \int_{M_0} H_0^2\,d\mu_0
=
n^2 |\mathbb{S}^n|^{{\frac{2}{n}}}
|M_0|^{\frac{n-2}{n}}.
\end{equation}
By Cauchy–Schwarz inequality and Minkowski inequality and \eqref{H02integral=}, we have
\begin{equation}
    \int_{M_0} H_0^2
=
\frac{\left(\int_{M_0} H_0\right)^2}{|M_0|}
=
n^2 |\mathbb{S}^n|^{{\frac{2}{n}}}
|M_0|^{\frac{n-2}{n}}.
\end{equation}
By the equality case of Minkowski inequality for outward minimizing closed hypersurface in \cite{HuiskenIlmanen2001} or \cite[Theorem 2 (b)]{FreireSchwartz2014}, $M_0$ must be a round sphere.  For the optimal parabolic estimates, it follows from  parabolic co-area formula in \cite[Theorem D, Example 3.5]{hershkovits2017isoperimetric}, optimal extinction time estimates and \eqref{toymtode}.

However, there are several delicate difficulties to apply the above argument for general mean curvature flows.  First of all, by taking the outward minimizing hull, the proof of Theorem \ref{Optimal extinction time estimates for weak mean curvature flow}  can be reduced to showing optimal extinction time estimates in \eqref{eq:general-extinction}, \eqref{eq:general-brakke-extinction} (see also Proposition  \ref{thm:LSF-Brakke extinction optimal}), but taking outward minimizing hull of initial smooth closed hypersurface of mean curvature flow may lose  regularity since the boundary of  outward minimizing hull has at most $C^{1, 1}$ regularity in the contact part and has codimension $7$ singularities in the free minimizing part. The regularity is important for applying extended Minkowski type inequality \cite[Theorem~1.1]{AFM22}, and it is also crucial for later quantitative analysis such as elliptic regularization. Secondly, the outward minimizing level set flows used in the proofs of Theorem \ref{Optimal extinction time estimates for weak mean curvature flow} and Theorem \ref{sharp parabolic estimates} generally have very complicated singularity structure. Hence, the area-dissipation
identity and the Minkowski inequality may fail across {those singular times}. Finally, for the equality case analysis, the extended Minkowski type inequality \cite[Theorem~1.1]{AFM22} we used does not include a rigidity or stability characterization, so we cannot fully apply the above toy model equality case strategy.

To overcome the first difficulty, we construct almost outward minimizing, strictly mean convex smooth approximation inspired by \cite{BZ26Penrose, BZ26CurvatureFree}. 
%{where certain regularity\footnote{More precisely, \cite[Lemma 5.5]{AFM22} assumes smoothness condition of $\partial\Omega$, and applying \cite[Theorem 1.1]{Schmidt2015} requires $\cH^n(\partial\Omega^*)=P(\Omega^*)$.} is required for this construction due to \cite[Lemma 5.5]{AFM22} and \cite[Theorem 1.1]{Schmidt2015}.}
For removing the second stumbling block, motivated by Schulze \cite{schulze2020optimal}, we apply Ilmanen's elliptic regularization \cite{Ilm94}, to obtain an approximate flow with improved regularity.
We then apply extended Minkowski inequality \cite{AFM22} to the time slices of the regularized flow and compare the translator curvature with the slice curvature. By taking a limit we establish Minkowski inequality for almost every time slice of the level set flow starting from mean convex approximations. To address the final obstacle in the equality case, we show that the mean curvature of suitable time slices of the approximating flows converges to a constant as the corresponding times tend to zero. Applying the weak Alexandrov theorem \cite{DM19}, together with the equality case to rule out multiple components, we conclude that the outward-minimizing hull of the initial domain is a round ball.
\subsection{Outline of proof}
In Section \ref{Preliminaries}, we discuss some preliminary setup and facts about  level-set flow, elliptic regularization, Brakke flow and outward minimizing property.  In Section \ref{timeslicesection}, we establish some approximation lemmas regarding the time-slices
of the elliptically regularized flow. In Section \ref{almostminkowski}, we derive the approximate Minkowski inequality for almost every time slice. In Section \ref{mean convex approximation}, we construct  almost outward
minimizing strictly mean convex approximating  flows used in  the proof of our main results. In Section \ref{proofofmainresults}, we use the above approximate flows to
conclude the proof of the optimal extinction time estimates and equality case in Theorem \ref{Optimal extinction time estimates for weak mean curvature flow}. In addition, we also prove the $L^{p}$-estimates of arrival time function of smooth mean convex domain in Theorem \ref{lpestimates arrival time} and show the optimal parabolic isoperimetric estimate for the space-time track of outward minimizing mean curvature flow in Theorem \ref{sharp parabolic estimates}. Finally, we also give some remarks and
propose some related problems.

\subsection*{Acknowledgements} This work was initiated in memory of Professor Tom Ilmanen during the 21st KIAS winter school on Differential Geometry (January 11 - 16, 2026). The foundational contributions of  Professor Tom Ilmanen to  mean curvature flow field and in particular the elliptic regularization method, theory of weak mean curvature flow through singularities and  the ideas of outward minimizing hulls arising from his joint work with Professor Gerhard Huisken on inverse mean curvature flow and the Riemannian Penrose inequality play an essential role  in our present work.
The first, third and last authors have been supported by the National Research Foundation of Korea (NRF) grant funded by the Korean
government (MSIT) NRF-2022R1C1C1013511 and  RS-2023-00219980 and by
Samsung Science \& Technology Foundation grant SSTF-BA2302-02. The second author was supported  by an AMS-Simons Travel
Grant while at the Massachusetts Institute of Technology, and is currently supported by Hunan University, and he also thanks Professor Felix Schulze for his encouragement regarding this project during the 2026 BIRS summer workshop on Geometric Flows and Related Topics in Banff.

\subsection*{AI disclosure} 
The writing and mathematical content of this paper are due to the authors.  All the crucial strategies  about taking outward minimizing hull, adapting elliptic regularization and constructing almost outward minimizing, strictly mean convex smoothing approximation are proposed by the authors. Besides the AI assistance with literature searches and language editing in this project, AI tools are used in analyzing the  convergence issue in the rigidity case discussion in the proof of Proposition  \ref{thm:LSF-Brakke extinction optimal}.

\section{Preliminaries}\label{Preliminaries}

The level-set formulation provides a weak continuation past
singularities by encoding the evolving hypersurfaces as level sets of a scalar
function.  This approach was developed independently by Evans--Spruck \cite{ES91} and by
Chen--Giga--Goto \cite{CGG91}, and it is especially effective for mean convex initial data,
where the flow moves monotonically inward.

Let $\Omega\subset \R^{n+1}$ be a bounded domain with smooth, strictly mean convex boundary
$M_0=\partial\Omega$.  The mean convex level set flow can be described by an
arrival-time function
\[
        u:\overline{\Omega}\to [0,\infty),
\]
where $u(x)$ is the time at which the evolving front reaches the point $x$.
With the sign convention used here, the time slices are
\[
        K_t=\{x\in \Omega : u(x) \geq t\}, \qquad M_t=\partial K_t,
\]
and, away from the singular set, the arrival time satisfies
\begin{equation}\label{eq:arrival}
        \operatorname{div}\!\left(\frac{Du}{|Du|}\right)
        =
        -\frac{1}{|Du|}
        \qquad \text{in } \Omega,
        \qquad
        u=0 \quad \text{on } \partial\Omega .
\end{equation}
As \eqref{eq:arrival} is a degenerate elliptic equation, elliptic regularization replaces \eqref{eq:arrival} by a family of uniformly
elliptic equations.  For $0<\eps\leq1$, let $u^\eps$ solve
\begin{equation}\label{eq:elliptic-reg}
        \operatorname{div}\!\left(
        \frac{Du^\eps}{\sqrt{\eps^2+|Du^\eps|^2}}
        \right)
        =
        -\frac{1}{\sqrt{\eps^2+|Du^\eps|^2}}
        \qquad \text{in } \Omega,
        \qquad
        u^\eps=0 \quad \text{on } \partial\Omega .
\end{equation}

From the strict mean convexity of $M_0$,
one can construct a barrier to deduce that $|Du^\eps|$ is bounded on $M_0$.
By the maximum principle, this gives an a priori gradient bound on $\overline \Omega$,
therefore \eqref{eq:elliptic-reg} is uniformly elliptic.
Applying the standard method of continuity, along with
De Giorgi--Nash--Moser and Schauder estimates yields a smooth solution.
As the gradient bound is uniform over $0<\eps\leq 1$,
by Arzela-Ascoli Theorem, every $\eps_i\to 0$ admits a subsequence such that
$u^{\eps_i}\to \tilde u$ uniformly on $\overline \Omega$
for some $\tilde u\in C^{0,1}(\overline \Omega)$.
Uniqueness gives a viscosity solution $u=\lim_{\eps\to0}u_\eps$ for the original problem \eqref{eq:arrival}. We summarize the convergence of $u^\eps$ to $u$ and other related quantities in the following lemma.

\begin{lemma}[{\cite[§3]{ES95}, see also \cite[Proposition 2.2 and Lemma 2.3]{MS08}}]\label{lem:elliptic-regularization-conv}
Let $\Omega\subset \R^{n+1}$ be a bounded domain with smooth, strictly mean convex boundary,
and let $u^\eps$ be the solutions of \eqref{eq:elliptic-reg}.  Let $u$ denote
the arrival-time function of the mean convex level set flow starting from
$\partial\Omega$, normalized by $u=0$ on $\partial\Omega$.  Then, as
$\eps\downarrow 0$,
\begin{enumerate}
    \item $u\in C^{0,1}(\overline{\Omega})$ and satisfies the arrival-time equation
    \eqref{eq:arrival} in the viscosity sense;

    \item \[u^\eps\to u \qquad
    \text{uniformly on } \overline{\Omega};\]

    \item
    \[
        |Du^\eps| \overset{*}{\rightharpoonup} |Du|
        \qquad \text{in } L^\infty(\Omega),
    \]
    that is, the integral converges when tested with a function in $L^1(\Omega)$;

    \item 
    \[
        \frac{Du^\eps}{\sqrt{\eps^2+|Du^\eps|^2}} \to \frac{Du}{|Du|}
        \qquad \text{in } L^p(\Omega;\R^{n+1})
    \]
    for every $1\le p<\infty$.
\end{enumerate}
\end{lemma}

The regularized equation has the following geometric interpretation.
For each $\eps>0$, the family of graphs $(N_t^\eps)$ given by
\[
        N_t^\eps
        =
        \left\{ \left(x,\frac{u^\eps(x)-t}{\eps}\right)
        : x\in \Omega \right\}
        \subset \Omega\times \mathbb{R}
\]
is a smooth translating mean curvature flow with translating speed $\eps^{-1}$.
Let
\[
        K_t=\{x\in\Omega:u(x) \geq t\},
        \qquad
        \mu_t:=\mathcal H^n\llcorner \partial^*K_t,
\]
\[
        \widetilde\mu_t
        :=
        \mu_t\otimes \mathcal L^1
        =
        \mathcal H^{n+1}\llcorner(\partial^*K_t\times\mathbb R).
\]
Roughly speaking, the translator $N_t^\eps$ elongates in the vertical direction,
and its mean curvature $H^\eps$ is close to the mean curvature of its horizontal slice.
Therefore as $\eps \rightarrow 0$ the translators converge to the level set 
with an additional $\R$ factor (namely $\partial^*K_t \times \R$) in a suitable sense.
The precise sense in which the translators converge to the level set flow is
the following:
\begin{lemma}[\cite{schulze2008nonlinear}]\label{lem:translator-convergence}
Given a sequence $\eps_i\downarrow0$, for
a.e. $t\ge0$,
\[
        \mathcal H^{n+1}\llcorner N_t^{\eps_i}
        \rightharpoonup
        \widetilde\mu_t
        \qquad
        \text{as Radon measures on } \Omega\times\mathbb R,
\]
and the rectifiable $n$-varifold associated with $\partial^*K_t$ carries a weak mean curvature vector
$\vec H_t \in L^2(\partial^*K_t, \cH^n)$ for a.e. t.
\end{lemma}

\begin{remark}\label{rmk:boundary-equivalence}
As noted in \cite[Section 2]{MS08}, since $u$ is Lipschitz, it also belongs to BV function class. Then we can apply co-area formula for Lipschitz functions and BV-functions. For large $T>\sup u$, 
\[ \int_0^T\cH^n(\partial^*K_t)dt=\int_\Omega|Du|dx=\int_0^T\cH^n(\{u=t\})dt.\]
Together with
\[\partial^*K_t \subset \{u=t\},\]
we conclude that the reduced boundary $\partial^*K_t$ agrees with the level set $M_t=\{u=t\}$ up to $\mathcal{H}^n$--null set for a.e.-$t$. A similar statement holds for the smooth functions $u^\eps$ for $\eps>0$.
\end{remark}

Now, we recall the definition of Brakke flow, which is a measure-theoretic approach to continue mean curvature flow through
singularities.
\begin{definition}[Integral Brakke flow cf. {\cite{Ilm94, SchulzeBrakkeFlows}}]\label{def-Brakke flow}
An $n\text{-dimensional}$  integral Brakke flow in $\mathbb{R}^{n+1}$
is a one-parameter family of Radon measures $(\mu_t)_{t\in I}$
over an interval $I\subset\mathbb{R}$ so that:
\begin{enumerate}
    \item For almost every $t\in I$, there exists an 
    $n$-dimensional integral varifold $V_t$ with
    \[
    \mu_t=\mu_{V_t}
    \]
    so that $V_t$ has locally bounded first variation, and has weak mean
    curvature vector $\vec H$ orthogonal to $\operatorname{Tan}(V(t),\cdot)$
    almost everywhere.

    \item For a bounded interval $[t_1,t_2]\subset I$ and any compact set
    $K\subset\mathbb{R}^{n+1}$,
    \[
    \int_{t_1}^{t_2}\int_K
    \left(1+|\vec H|^2\right)\,d\mu_t\,dt<\infty.
    \]

    \item If $[t_1,t_2]\subset I$ and
    \[
    f\in C_c^1\bigl(\mathbb{R}^{n+1}\times[t_1,t_2]\bigr),
    \qquad f\geq 0,
    \]
    then we have Brakke's inequality 
    \begin{align}\label{Brakke inequality}
         &\quad\int f(\cdot,t_2)\,d\mu_{t_2}
    -
    \int f(\cdot,t_1)\,d\mu_{t_1}\\\notag
    &\leq
    \int_{t_1}^{t_2}\int
    \left(
        -|\vec H|^2 f
        +\langle\vec H,\nabla f\rangle
        +\frac{\partial f}{\partial t}
    \right)
    d\mu_t\,dt.
    \end{align}
\end{enumerate}
\end{definition}

Compared with level set flow, Brakke flow is not necessarily unique and has potential mass drop. A nice property of mean convex level set flow is that given a mean convex level set flow $\{K_t\}$, the flow
\begin{equation}
\mu_t=\cH^n\llcorner\partial^*K_t
\end{equation}
becomes a unit density Brakke flow, where $\partial^*K_t$ is the reduced boundary of $K_t$.
Metzger-Schulze later proved that this is a Brakke flow with no mass drop, and the defining Brakke's inequality holds with equality \cite{MS08}.
\begin{lemma}[cf. {\cite{MS08}}]\label{lem:no-mass-drop} 
Let $\Omega\subset\R^{n+1}$ be a bounded open set whose boundary is a closed $C^1$ hypersurface carrying a nonnegative weak mean curvature in $L^2$. If $(\mu_t)_{t\geq0}$ is the family of Radon measures associated with the level set flow of $\partial\Omega$, for $\phi\in C_c^2(\mathbb{R}^{n+1})$, $\phi\ge0$, and $0\le t_1<t_2$, we have
\[
        \mu_{t_2}(\phi)-\mu_{t_1}(\phi)
        =
        \int_{t_1}^{t_2}\int
        \left(
            -\phi |\vec H|^2+\langle \nabla \phi,\vec H\rangle
        \right)
        \,d\mu_t\,dt,
\]
i.e. the defining inequality of Brakke flow (Brakke's inequality) holds with equality.
\end{lemma}
In particular, by choosing $\phi\equiv1$ on the compact support of the flow $\{M_t\}=\{\partial^*K_t\}$, one obtains
the first variation of area
\begin{equation}\label{mean convex level set flow schulze}
    \frac{d}{dt}\mathcal{H}^n(M_t)
        =
        -\int_{M_t} |\vec H|^2\,d\mathcal{H}^n
        \qquad \text{for a.e. } t.
\end{equation}

Regarding the first variation of area, mean convex hypersurfaces can be understood as local one-sided minimizers of area. Among finite-perimeter sets, we can consider the global one-sided minimizing conditions defined as follows.

\begin{definition}
    A set $E\subset\R^{n+1}$ of finite perimeter is outward-minimizing if for every bounded finite perimeter set $F \supset E$,
    $P(F) \geq P(E)$. $E$ is strictly outward-minimizing if $P(F)=P(E)$ implies $E=F$ up to an $\cH^{n+1}$-null set.
\end{definition}
We note that the outward minimizing property is preserved by level set flow by \cite[Theorem 3.5]{Whi00}. We also have the following definition of approximate outward minimizing property.
\begin{definition} \label{def:almostoutwardmin}
    A set $E\subset\R^{n+1}$ of finite perimeter is $\eta$--outward--minimizing if for every bounded finite perimeter set $F \supset E$,
    $P(F) \geq P(E)-\eta$.
\end{definition}

Given a set $E$ of finite perimeter, consider the problem of minimizing perimeter among finite perimeter sets that contain $E$.
Existence of such a minimizer follows by compactness of BV functions, and it is immediate that the minimizers are outward-minimizing.
In addition, we know that there is a unique outermost minimizer up to a null set (see \cite{FM22}),
allowing us to give the following definition:
\begin{definition}
    The minimizing hull $E^*$ of a finite perimeter set $E$ is a strictly outward-minimizing bounded finite-perimeter set that contains $E$ and satisfies
    \[P(E^*)\leq P(F)\]
    for all $F\supset E$.
\end{definition}

\section{Time-slice analysis of elliptic regularization}\label{timeslicesection}

The goal of this section is to establish lemmas regarding the time-slices of the regularized flow that will be used in the proof of the main theorem.
Following the notation of the previous section, we denote by $H$ the mean curvature of the level set flow,
and $H^\eps$ the mean curvature of the translator $N_t^\eps\subset\R^{(n+1)+1}$.
Since $N_t^\eps$ with varying $t$ are all congruent, we will use the same notation $H^\eps$ for the mean curvature of $N_t^\eps$ for all $t\ge 0$.
Considering that $N_0^\eps$ is an elongated graph of $u^\eps$, we can regard $H^\eps$ as a function on $\overline\Omega$ by taking the composition of $H^\eps$ with the map $x\mapsto (x, u^\eps(x)/\eps)$.
With this identification, we have
\begin{equation}\label{eq:Hepsilon}
H^\eps=\frac{1}{\sqrt{\eps^2+|Du^\eps|^2}}.
\end{equation}

\begin{proposition}\label{cor:area-convergence}
    There exists a sequence $\eps_i \rightarrow 0$ such that for a.e. $t\ge 0$,
    \[
        \cH^n(\{u^{\eps_i}=t\})\to\cH^n(\{u=t\})
         \quad \text{as } i\to\infty.
    \]
\end{proposition}
\begin{proof}  
    By the co-area formula, for $0\le a<b$,
    \[ \int_a^b \cH^n(\{u^\eps=t\}) dt = \int_{\{a<u^\eps<b\}} |Du^\eps| dx,\]
    and
    \[ \int_a^b \cH^n(\{u=t\}) dt = \int_{\{a<u<b\}} |Du| dx.\]
    By Lemma \ref{lem:elliptic-regularization-conv},
    \[|Du^\eps| \overset{*}{\rightharpoonup} |Du| \quad \text{in } L^\infty(\Omega),\]
    and
    \[\chi_{\{a<u^\eps<b\}} \rightarrow \chi_{\{a<u<b\}} \quad \text{in } L^1(\Omega),\]
    because $u^{\eps_i}$ converges to $u$ uniformly and the level sets of $u$ have zero $\mathcal{L}^{n+1}$--measure.
    This implies that
    \[ \int_{\{a<u^\eps<b\}} |Du^\eps| dx \rightarrow 
    \int_{\{a<u<b\}} |Du| dx,\]
    and equivalently
    \[\int_a^b \cH^n(\{u^\eps=t\}) dt \rightarrow \int_a^b \cH^n(\{u=t\}) dt.\]
    Here, by lower semicontinuity of total variation for BV functions and Remark \ref{rmk:boundary-equivalence},
    \[\chi_{\{u^\eps>t\}} \rightarrow \chi_{\{u>t\}} \quad \text{in } L^1(\Omega)\]
    gives
    \[\cH^n(\{u=t\}) \leq \liminf_{\eps\to 0} \cH^n(\{u^\eps=t\})\]
    for a.e.-$t$. Therefore, $\cH^n(\{u^\eps=t\}) \rightarrow \cH^n(\{u=t\})$ in $L^1([0,T])$ for any $T>0$.
    We conclude the proof by extracting a subsequence $\eps_i \rightarrow 0$ such that
    \[\cH^n(\{u^{\eps_i}=t\})\to\cH^n(\{u=t\}) \quad \text{for a.e. } t\ge 0.\]
\end{proof}
Also, we have the following convergence of integrated mean curvature.
\begin{proposition}\label{lem:mean-curv-convergence}
    \[
        \int_{\{u^{\eps}=t\}} H^{\eps}d\mathcal H^{n}\,dt
        \rightharpoonup
        \int_{\{u=t\}} H d\mathcal H^n\,dt
    \]
    as measures in $t$.
\end{proposition}
\begin{proof}
For $\psi \in C([0,T])$, since $H^\eps = \frac{1}{\sqrt{\eps^2+|Du^\eps|^2}}$ we have
\[
\int\psi(t)\int_{\{u^{\eps}=t\}} H^{\eps}d\mathcal H^{n}\,dt
= \int_{\Omega}\psi(u^{\eps})\frac{|Du^{\eps}|}{\sqrt{\eps^2+|Du^{\eps}|^2}} dx
\]
by the co-area formula. Similarly,
\[
\int\psi(t)\int_{\{u=t\}}Hd\mathcal H^{n}\,dt
= \int_{\Omega}\psi(u) dx
\]
by Lemma 4.4 of \cite{MS08}.
The convergence of the right hand side follows from Lemma \ref{lem:elliptic-regularization-conv}.

\end{proof}

Using the ellipticity of the regularized equation,
one can show that on $N^\eps \coloneq N^\eps_0$, the second fundamental form can be
bounded by its mean curvature uniformly in $0<\eps<\eps_0$.

\begin{lemma}%[Slice maximum principle for translators]
\label{trans slice mp}
Let \(M^n\subset \mathbb{R}^{n+1}\) be a strictly mean convex
translator whose velocity is a positive scalar multiple of
\(-e_{n+1}\). Then, the quantity $|A|/H$ cannot attain an interior local maximum unless it is constant, where $A$ denotes the second fundamental form and $H$ denotes the mean curvature.
In particular, if $M\cap\{x_{n+1}\ge h\}$ is compact, we have
\[
\sup_{x\in M\cap\{x_{n+1}\ge h\}}
\frac{|A|}{H}
=
\sup_{x\in M\cap\{x_{n+1}=h\}}
\frac{|A|}{H}.
\]
\end{lemma}

\begin{proof}
Let
\[
V=-\lambda e_{n+1},
\qquad \lambda>0,
\]
be the translating velocity. We use the convention that the mean curvature vector is
\[
\vec H=-H\nu.
\]
where $\nu$ is the unit normal with $\langle \nu,e_{n+1}\rangle > 0$.
Thus the translator equation \(\vec H=V^{\perp}\) becomes
\[
H=-\langle V,\nu\rangle.
\]
Since \(M\) is strictly mean convex, \(H>0\) on \(M\).

Define the drift Laplacian
\[
\mathcal{L}_{V}
:=
\Delta_{M}+\langle V,\nabla \,\cdot\,\rangle.
\]
We claim that
\[
\mathcal{L}_{V}A=-|A|^2A
\]
and by taking the trace
\[
\mathcal{L}_{V}H=-|A|^2H.
\]
Indeed, at a point where the chosen orthonormal frame is geodesic,
\[
\nabla_jH=-\overline\nabla_j\langle V, \nu\rangle=-A_{jk}V_k
\]
where $\overline\nabla$ denotes the Euclidean connection of the ambient space. Differentiating again,
\begin{align*}
\nabla_i\nabla_jH
&=-(\nabla_iA_{jk})V_k-A_{jk}\nabla_iV_k\\
&=
-V_k\nabla_k A_{ij}
-HA_{ik} A_{jk}.
\end{align*}
For the second line we have used the Codazzi identity $\nabla_iA_{jk}=\nabla_kA_{ij}$ and
\[\nabla_iV_k=\overline\nabla_i\langle V, \partial_k\rangle=-A_{ik}\langle V, \nu\rangle=HA_{ik}.\]
Combining this with the standard identity (see \cite[Lemma 2.1]{Hui84})
\[
\Delta_{M} A_{ij}
=
\nabla_i\nabla_jH
+HA_{ik} A_{kj}
-|A|^2A_{ij}
\]
yields the claimed equation for \(A\).

Now set
\[
B:=\frac{A}{H}.
\]
Using the two equations above and the product rule, we obtain
\[
\left(
\mathcal{L}_{V}
+
2\langle \nabla\log H,\nabla\,\cdot\,\rangle
\right)B=0.
\]
Consequently, for
\[
q:=|B|^2=\frac{|A|^2}{H^2},
\]
we have
\begin{align*}
\left(
\mathcal{L}_{V}
+
2\langle \nabla\log H,\nabla\,\cdot\,\rangle
\right)q
&=
2|\nabla B|^2 \\
&=
\frac{2}{H^4}
\left|
H\nabla A-A\otimes\nabla H
\right|^2
\ge 0.
\end{align*}

Thus \(q\) is subharmonic with respect to the elliptic operator
\[
\mathcal{L}
:=
\Delta_{M}
+\langle V,\nabla\,\cdot\,\rangle
+2\langle \nabla\log H,\nabla\,\cdot\,\rangle.
\]
Since \(H>0\), the coefficients of
\(\mathcal{L}\) are smooth and bounded on compact subsets of $M$. Therefore, the
strong maximum principle applies, and \(q\) cannot attain a local maximum in the interior of \(M\), unless $q$ is constant. In particular, for any \(h\in\mathbb{R}\),
\[
\sup_{x\in M\cap\{x_{n+1}\ge h\}}
\frac{|A|}{H}
=
\sup_{x\in M\cap\{x_{n+1}=h\}}
\frac{|A|}{H}.
\]
\end{proof}

\begin{lemma} \label{pinching}
Let $\Omega$ be a bounded domain with smooth, strictly mean convex boundary, and $\eps_i>0$ be a sequence of numbers decreasing to $0$.
Given $\tau>0$ and $\alpha<0$,
there are large $C=C(\tau)>0$ and large $N(\tau,\alpha)\in\mathbb{N}$ such that for $t\geq\tau$ and $i\geq N(\tau,\alpha)$,
\[
    |A^{\eps_i}| \leq C H^{\eps_i} \quad \text{on} \quad N_t^{\eps_i}\cap(\Omega\times [\alpha,\infty)), 
\]
where $A^{\eps_i}$ and  $H^{\eps_i}$ denote the second fundamental form and mean curvature, respectively, for $N^{\eps_i}_t$.
\end{lemma}
\begin{proof}

Let $\{K_t\}_{t\geq0}$ be the level set flow with $K_0=\Omega$. Since \(\partial\Omega\) is smooth,
for a short time the evolution of $\partial\Omega$ given by $\{\partial^*K_t\}_{t\geq0}$ is a smooth mean curvature flow. Therefore, by the convergence given in Lemma \ref{lem:translator-convergence} and the Brakke regularity theorem (see \cite[Section 6]{Bra78} and \cite[Section 12]{Ilm94}), there is $0<t_0<\frac{\tau}{2}$ such that 
\[
N_{t_0}^{\eps_i} \rightarrow \partial^*K_{t_0}\times \R \quad \text{in} \ C^\infty_\text{loc}.
\]
Hence, there is a large enough $N(\tau,\alpha)\in\mathbb{N}$ such that for $i\geq N(\tau,\alpha)$,
\[\eps_i<\frac{\tau}{-2\alpha}\quad \text{and}\quad
\sup_{ N^{\eps_i}_{t_0}\cap(\Omega\times\{0\})}
\frac{|A^{\eps_i}|}{H^{\eps_i}}
<
\sup_{\partial^*K_{t_0}}
    \frac{|A_{\partial^*K_{t_0}}|}{H_{\partial^*K_{t_0}}}
+
1=:C(\tau).
\]
Note that since $\partial\Omega$ is mean convex, $\partial^*K_{t_0}$ is a strictly mean convex smooth hypersurface.
Above constant is therefore finite. Here, an important observation is that for $t\geq\tau$, $N^{\eps_i}_t$ is a translate of $N^{\eps_i}_{t_0}$, and
the slice $N_{t_0}^{\eps_i}\cap(\Omega\times\{0\})$ corresponds to $N^{\eps_i}_t\cap(\Omega\times\{z(i,t)\})$ with
\[ z(i,t):=-\frac{t-t_0}{\eps_i}<-\frac{\tau}{2\eps_i}<\alpha\quad\text{for}\;i\geq N(\tau,\alpha).\]
Therefore by Lemma \ref{trans slice mp}, for $t\geq\tau$ and $i\geq N(\tau,\alpha)$,
\begin{align*}
&\sup_{N^{\eps_i}_t\cap(\Omega\times [\alpha,\infty))}
\frac{|A^{\eps_i}|}{H^{\eps_i}}
\le
\sup_{N^{\eps_i}_t\cap(\Omega\times [z(i,t),\infty))}
\frac{|A^{\eps_i}|}{H^{\eps_i}}
\le
\sup_{N^{\eps_i}_t\cap(\Omega\times \{z(i,t)\})}
\frac{|A^{\eps_i}|}{H^{\eps_i}}\\
&=\sup_{ N^{\eps_i}_{t_0}\cap(\Omega\times\{0\})}
\frac{|A^{\eps_i}|}{H^{\eps_i}}<C(\tau).    
\end{align*}

This proves the desired estimate.
\end{proof}

\section{Approximate Minkowski's inequality for almost every slice}\label{almostminkowski}

\begin{lemma}[Relaxed Minkowski's inequality]\label{lem:relaxed-minkowski}
    Suppose $\eps>0$ is given, and assume that $\Omega\in\R^{n+1}$ is a bounded, open $\eta$--outward--minimizing set (see Definition \ref{def:almostoutwardmin}) with smooth, strictly mean convex boundary. For a.e. $s>0$, $\Sigma_s^\eps = \{u^\eps=s\}$ satisfies 
    \begin{equation}\label{eq:relaxed-minkowski}
\int_{\Sigma_s^\eps} |H_s^\eps| \geq n|\mathbb{S}^n|^{\frac{1}{n}}\left(|\Sigma_s^\eps|-\delta_s^\eps-\eta\right)^{\frac{n-1}{n}}_+,
\end{equation}
where $H^{\eps}_s$ denotes the mean curvature of $\Sigma^{\eps}_s$,
$\delta_s^\eps\geq 0$, and 
\begin{equation}\label{eq:delta-bound}
\int_{\mathbb R}\delta_{s}^\eps\,ds
\leq
\eps|\Omega|.
\end{equation}

\end{lemma}
\begin{proof}
By Sard's theorem, for a.e. $s>0$, $\Sigma_s^\eps$ is a smooth hypersurface. For such $s$, we consider the vector field
\[ X^\eps = -\frac{Du^\eps}{\sqrt{\eps^2+|Du^\eps|^2}}.\]
This is a vector field that is smooth up to the boundary.
We have $|X^\eps| \leq 1$ and \[\operatorname{div}(X^\eps) = \frac{1}{\sqrt{\eps^2+|Du^\eps|^2}}>0.\]
Let $E_s^\eps = \{u^\eps > s\}$, and let $F$ be a set with finite perimeter such that $E_s^\eps \subset F \subset \Omega$.
If we let $\nu^\eps=-Du^\eps/|Du^\eps|$ be the outward unit normal of $E_s^\eps$ and $\nu_F$ be the unit normal of $F$,
by extending $X^\eps$ smoothly to a neighborhood of $\overline \Omega$ and applying the divergence theorem we obtain
\begin{align}\label{eq:divergence}
    0\leq\int_{F\setminus E_s^\eps} \operatorname{div}(X^\eps) &= \int_{\partial^* F} \langle X^\eps, \nu_F\rangle - \int_{\Sigma_s^\eps} \langle X^\eps, \nu^{\eps}\rangle \\
     &\leq P(F) - \int_{\Sigma_s^\eps} \frac{|Du^\eps|}{\sqrt{\eps^2+|Du^\eps|^2}}.
\end{align}
Therefore, we have
\be\label{eq:perimeter-bound} P(E_s^\eps) \leq P(F) + \delta_s^\eps.\ee
with
\[\delta_s^\eps =\int_{\Sigma_s^\eps} \left(1-\frac{|Du^\eps|}{\sqrt{\eps^2+|Du^\eps|^2}}\right) \geq 0\]
Let \(E_{s}^{\eps *}\) be the minimizing hull of \(E_{s}^\eps\).
Since
\[
E_{s}^\eps\subset E_{s}^{\eps *}\cap\Omega,
\]
by \eqref{eq:perimeter-bound} we have
\[ P(E_s^\eps) \leq P(E_{s}^{\eps *}\cap\Omega) + \delta_s^\eps.\]

By a standard inequality for perimeters, (cf. \cite[Lemma 12.22]{Mag12}), 
\begin{equation}\label{eq:submodularity}
P(E_{s}^{\eps *}\cap\Omega)
+
P(E_{s}^{\eps *}\cup\Omega)
\leq
P(E_{s}^{\eps *})
+
P(\Omega).
\end{equation}
Since $E_{s}^{\eps *}\cup\Omega
\supset\Omega$, $\eta$--outward-minimizing property of $\Omega$ gives
\[
P(\Omega)
\leq
P(E_{s}^{\eps *}\cup\Omega)
+\eta
\]
Substituting this into \eqref{eq:submodularity}, we obtain
\[
P(E_{s}^{\eps *}\cap\Omega)
\leq
P(E_{s}^{\eps *})+\eta.
\]
Together with \eqref{eq:perimeter-bound} and Remark \ref{rmk:boundary-equivalence}, this yields
\begin{equation}\label{eq:hull-lower-bound}
P(E_{s}^{\eps *})
\geq P(E_{s}^\eps)-\delta_{s}^\eps-\eta
= |\Sigma_{s}^\eps|-\delta_{s}^\eps-\eta.
\end{equation}

Applying the extended Minkowski inequality
\cite[Theorem~1.1]{AFM22}   on $E_{s}^{\eps *}$ ($n\geq 2$) and Fenchel's total absolute curvature inequality $\int_{\Sigma_{s}^{\varepsilon}}|\kappa| ds\geq 2\pi$ ($n=1$), together with
\eqref{eq:hull-lower-bound}, we prove \eqref{eq:relaxed-minkowski}: 
\begin{align*}
\int_{\Sigma^{\eps}_{s}}
|H_{s}^\eps|\,d\mathcal H^n
&\geq
n|{\mathbb{S}^n}|^{1/n}
P(E_{s}^{\eps *})^{\frac{n-1}{n}} \geq
n|{\mathbb{S}^n}|^{1/n}
\left(
|\Sigma_{s}^\eps|
-\delta_{s}^\eps
-\eta
\right)^{\frac{n-1}{n}}_+.
\end{align*}

To show \eqref{eq:delta-bound}, we integrate $\delta_s^\eps$ in $s$ and apply the co-area formula:
\begin{align}
    \int_{-\infty}^\infty \delta_s^\eps ds &= \int_{-T}^T \int_{\Sigma_s^\eps} \left(1-\frac{|Du^\eps|}{\sqrt{\eps^2+|Du^\eps|^2}}\right) d\mathcal H^n ds \\
    &= \int_\Omega \left(1-\frac{|Du^\eps|}{\sqrt{\eps^2+|Du^\eps|^2}}\right) |Du^\eps| dx \\
    &\leq \int_\Omega \eps dx = \eps|\Omega|.
\end{align}
\end{proof}

\begin{lemma}\label{strong slice conv}
Let $\Omega$ be a bounded domain with smooth, strictly mean convex boundary, and $0<\eps_i\rightarrow+0$. Given $\varphi\in C_c((0,T_{\mathrm{ext}}))$, we have
\[
\lim_{i\rightarrow\infty}\int \varphi(s)\int_{\Sigma_s^{\eps_i}}
\bigl(H^{\eps_i}-|H_s^{\eps_i}|\bigr)\,d\mathcal H^n\,ds=0,
\]
\end{lemma}

\begin{proof}
First, by the co-area formula, for $\eps>0$ we have
\begin{align}\label{eq:mean-curvature-co-area}
    \int \varphi(s)\int_{\Sigma_s^\eps}
\bigl(H^\eps-|H_s^\eps|\bigr)\,d\mathcal H^n\,ds\!=\!\int_{\Omega} \varphi(u^{\eps}(x))
\bigl(H^{\eps}(x)-|H_{u^{\eps}(x)}^\eps(x)|\bigr)|Du^{\eps}(x)|dx.
\end{align}

We denote the singular set of the level set flow $\tilde S \subset \operatorname{graph}u$.
From \cite{Whi00}, we know that the parabolic Hausdorff dimension of $\tilde S$ is at most $n-1$.
As in \cite[Lemma 3.3]{MS08}, the projection $S=\Pi(\tilde S)$ of $\tilde S$ onto $\Omega$ is a $\cH^n$-null set.
We claim the pointwise convergence:
\[\lim_{i\rightarrow\infty} \bigl(H^{\eps_i}(x)-H_{u^{\eps_i}(x)}^{\eps_i}(x)\bigr)=0\;\text{for}\; x\in\Omega\setminus S.\]

Fix $x_0\in\Omega\setminus S.$ Since $x_0\in\Omega\setminus S,$
in a spacetime neighborhood of $(x_0,u(x_0))\in\operatorname{graph}u$,
the level set flow is a smooth mean curvature flow. Therefore,
it follows from Lemma \ref{lem:translator-convergence} and the Brakke regularity theorem (cf. \cite[Section 6]{Bra78}, \cite[Section 12]{Ilm94})
that there is an open neighborhood $U\subset\Omega$ of $x_0$ such that
$N_t^{\eps_i}$ converges to $\Sigma_t\times\R$ smoothly
in any compact subset of $(\Sigma_t \cap U) \times \R$.
Therefore, both $H^{\eps_i}(x_0)$ and $H_{u^{\eps_i}(x_0)}^{\eps_i}(x_0)$ converge to the same limit,
which is the mean curvature of the smooth mean curvature flow at $x_0$.
This proves the claim. 

We note that $|Du^{\eps_i}|$ is bounded in $L^\infty(\Omega)$
by Lemma \ref{lem:elliptic-regularization-conv}-(3) and the uniform boundedness principle. Moreover, we have
\begin{equation}\label{measure zero}
\cH^n\left(\{x\in\Omega:\lim_{i\rightarrow\infty} \bigl(H^{\eps_i}(x)-H_{u^{\eps_i}(x)}^{\eps_i}(x)\bigr)\neq0\}\right)\leq \cH^n(S)=0.
\end{equation}
Therefore, the previous claim implies that the integrand on the right hand side in \eqref{eq:mean-curvature-co-area}
converges to zero almost everywhere in $\Omega$.

Next, since $\varphi\in C_c((0,T_{\mathrm{ext}}))$, there is a small $\tau>0$ such that $\varphi(s)=0$ for $s\leq\tau$.
We can further shrink $\tau$ so that the level set flow $\{M_t\}_{t\in[0,\tau]}$ with $M_0=\partial\Omega$ is a smooth mean curvature flow for $t\in[0,\tau]$.
Then $M_\tau$ is a mean convex hypersurface.
By Lemma \ref{pinching}, there exist large $C>0$ and $N\in\mathbb{N}$ such that for $t\geq\tau$ and $i\geq N$,
\begin{equation*}
    |A^{\eps_i}| \leq C H^{\eps_i} \quad \text{on} \quad N_t^{\eps_i}\cap(\Omega\times [0,\infty)), 
\end{equation*}

Given a regular point $x\in\Omega$ of $u^{\eps_i}$ with $i\geq N$ and $u^{\eps_i}(x)\geq\tau$, choose a local orthonormal frame
$e_1,\dots,e_n$ of $\Sigma_{u^{\eps_i}(x)}^{\eps_i}$ centered at $x$. These vectors are horizontal
tangent vectors on $N^{\eps_i}$, and hence at $x$ we have
\begin{align*}
|H^{\eps_i}_s||Du^{\eps_i}|&=\Bigr|\sum^n_{j=1}\frac{D^2u^{\eps_i}(e_j,e_j)}{|Du^{\eps_i}|}\Bigr||Du^{\eps_i}|\\
&=\Bigr|\sum^n_{j=1}\frac{D^2u^{\eps_i}(e_j,e_j)}{\sqrt{\eps_i^2+|Du^{\eps_i}|^2}}\Bigr|\sqrt{\eps^2+|Du^{\eps_i}|^2}\\
&\leq n|A^{\eps_i}|\sqrt{\eps_i^2+|Du^{\eps_i}|^2}\\
&\leq nCH^{\eps_i}\sqrt{\eps_i^2+|Du^{\eps_i}|^2}=nC.
\end{align*}
In the last equality, we used \eqref{eq:Hepsilon}.

For points $x\in\Omega$ with $u^{\eps_i}(x)<\tau$, we have $\varphi(u^{\eps_i}(x))=0$.
Thus, by Sard's theorem, the integrand of right hand side in \eqref{eq:mean-curvature-co-area} is bounded almost everywhere.
By the dominated convergence theorem, we conclude that
\[
\lim_{i\rightarrow\infty}\int \varphi(s)\int_{\Sigma_s^{\eps_i}}
\bigl(H^{\eps_i}-|H_s^{\eps_i}|\bigr)\,d\mathcal H^n\,ds=0.
\]
\end{proof}

\begin{proposition}[Minkowski inequality for almost every time slice]\label{thm:ae-minkowski}
Let $\Omega\subset \mathbb R^{n+1}$ be a bounded, open $\eta$--outward--minimizing set with smooth, strictly mean convex boundary.
Let $\{K_s\}_{0\le s<T_{\rm ext}}$ be the mean convex level set flow
with $K_0=\Omega$, and set
\[
    \Sigma_s:=\partial^*K_s,\qquad
    \mu_s:=\mathcal H^n\llcorner \Sigma_s .
\]
Then, for a.e. $s\in (0,T_{\rm ext})$,
\[
    \int_{\Sigma_s} H\,d\mu_s
    \ge
    n|{\mathbb{S}^n}|^{1/n}\left(|\Sigma_s|-\eta\right)^{\frac{n-1}{n}}_+ .
\]
\end{proposition}

\begin{proof}
When $n=1$, we by the strictly mean convex boundary condition on $\partial \Omega$, the level set flow agrees with the smooth planar curve shortening flow, which remains smooth until it extincts, so $\partial K_t=\partial^* K_t=\Sigma_s$. Moreover based on the number of convex boundary components of $\Sigma_s$, by turning-angle theorem, we always have
\begin{equation}
    \int_{\partial \Omega} H d\mu_s= \int_{\partial \Omega} \kappa d\mu_s\geq 2\pi=|\mathbb{S}^1|,
\end{equation}
which verifies \eqref{thm:ae-minkowski} when $n=1$. Hence, we can assume $n\geq 2$ below.

Let $u^\eps$ be the elliptic regularization of the arrival time,
and write
\[
    \Sigma_s^\eps:=\{u^\eps=s\},\qquad
    d\mu_s^\eps:=d\mathcal H^n\llcorner \Sigma_s^\eps .
\]

Fix a nonnegative function
\[
    \varphi\in C_c^\infty((0,T_{\rm ext})).
\]
By Proposition \ref{cor:area-convergence}, there is a subsequence $\eps_i\searrow 0$ such that the slice
area convergence
\begin{equation}\label{eq:area-convergence}
    |\Sigma_s^{\eps_i}|\rightarrow |\Sigma_s|
\end{equation}
holds for a.e. $s\in (0,T_{\rm ext})$. By Proposition \ref{lem:mean-curv-convergence},
\[
    \int \varphi(s)
        \left(\int_{\Sigma_s} H\,d\mu_s\right)ds
    =
    \lim_{i\to\infty}
    \int \varphi(s)
        \left(\int_{\Sigma_s^{\eps_i}}
            H^{\eps_i}\,d\mu_s^{\eps_i}
        \right)ds .
\]
Moreover, by Lemma \ref{strong slice conv},
\[
    \lim_{i\rightarrow\infty}\int \varphi(s)
    \left[
        \int_{\Sigma_s^{\eps_i}}
        \bigl(|H_s^{\eps_i}|-H^{\eps_i}\bigr)
        \,d\mu_s^{\eps_i}
    \right]ds = 0 .
\]
Therefore,
\[
    \int \varphi(s)
        \left(\int_{\Sigma_s} H\,d\mu_s\right)ds
    =
    \lim_{i\to\infty}
    \int \varphi(s)
        \left(\int_{\Sigma_s^{\eps_i}}
            |H_s^{\eps_i}|\,d\mu_s^{\eps_i}
        \right)ds .
\]

For a.e. $s$, Lemma \ref{lem:relaxed-minkowski}, the relaxed Minkowski inequality, applied to
$\Sigma_s^{\eps_i}$ gives
\[\int_{\Sigma_s^{\eps_i}}
        |H_s^{\eps_i}|\,d\mu_s^{\eps_i}
    \ge
    n|{\mathbb{S}^n}|^{1/n}
    \left(
        |\Sigma_s^{\eps_i}|-\delta_s^{\eps_i}-\eta
    \right)^{\frac{n-1}{n}}_+
    \]
    where $\delta_s^{\eps_i}\ge 0 \;\text{and}\;\|\delta_s^{\eps_i}\|_{L^1(\mathbb R)}\leq
    \eps_i|\Omega|.$

Then, Jensen's inequality gives
\[
\begin{aligned}
    \int \varphi(s)
        \left(\int_{\Sigma_s} H\,d\mu_s\right)ds
    &\ge
    \liminf_{i\to\infty}
    n|{\mathbb{S}^n}|^{1/n}
    \int \varphi(s)
    \left[
        \left(|\Sigma_s^{\eps_i}|-\eta\right)_+^{\frac{n-1}{n}}
        -
        (\delta_s^{\eps_i})^{\frac{n-1}{n}}
    \right]ds .
\end{aligned}
\]

We now remove the error term in the limit.
By Hölder's inequality,
\begin{align*}
    \int
        \varphi(s)(\delta_s^{\eps_i})^{\frac{n-1}{n}}\,ds
    &\le
    \|\varphi\|_{L^\infty}
        \int_{0}^{T_{\rm ext}} (\delta_s^{\eps_i})^{\frac{n-1}{n}}\,ds\\
    &\le
    \|\varphi\|_{L^\infty}
    (T_{\rm ext})^{1-\frac{n-1}{n}}
    \left(
        \int_{0}^{T_{\rm ext}} \delta_s^{\eps_i}\,ds
    \right)^{\frac{n-1}{n}}\\
    &\le
    T_{\rm ext}\|\varphi\|_{L^\infty}
    (\eps_i)^{\frac{n-1}{n}}
    \rightarrow 0
\end{align*}
as $i\to\infty$.
   
Hence,
\[
    \int \varphi(s)
        \left(\int_{\Sigma_s} H\,d\mu_s\right)ds
    \ge
    n|{\mathbb{S}^n}|^{1/n}
    \liminf_{i\to\infty}
    \int \varphi(s)
        \left(|\Sigma_s^{\eps_i}|-\eta\right)^{\frac{n-1}{n}}_+\,ds .
\]
Finally, by \eqref{eq:area-convergence} and Fatou's lemma, we get
\[
    \liminf_{i\to\infty}
    \int \varphi(s)
        \left(|\Sigma_s^{\eps_i}|-\eta\right)^{\frac{n-1}{n}}_+\,ds
    \ge
    \int \varphi(s)\left(|\Sigma_s|-\eta\right)^{\frac{n-1}{n}}_+\,ds .
\]
Therefore,
\[
    \int \varphi(s)
        \left(\int_{\Sigma_s} H\,d\mu_s\right)ds
    \ge
    \int \varphi(s)
        n|{\mathbb{S}^n}|^{1/n}
        \left(|\Sigma_s|-\eta\right)^{\frac{n-1}{n}}_+\,ds .
\]
Since 
$\varphi\in C_c^\infty((0,T_{\rm ext}), \mathbb{R}_+)$ is arbitrary, it follows that
\[
    \int_{\Sigma_s} H\,d\mu_s
    \ge
    n|{\mathbb{S}^n}|^{1/n}
    \left(|\Sigma_s|-\eta\right)^{\frac{n-1}{n}}_+
\]
for a.e. $s\in (0,T_{\rm ext})$.
\end{proof}

Above theorem becomes ineffective for time slices with small area, namely when \(|\Sigma_s|\leq \eta\).
For such slices, we need the following  estimate.
\begin{lemma}\label{lem:MS-inequality} Let $\Omega\subset \mathbb R^{n+1}$ be a domain with smooth, strictly mean convex boundary. Let $\{K_s\}_{0\le s<T_{\rm ext}}$ be the mean convex level set flow
with $K_0=\Omega$.
Then, for a.e. $s\in (0,T_{\rm ext})$,
\[
    \int_{\Sigma_s} H\,d\mu_s
    \ge
    c_{\mathrm{MS}}|\Sigma_s|^{\frac{n-1}{n}},
\]
where $c_{\mathrm{MS}}(n)=n|{\mathbb{B}^n}|^{\frac{1}{n}}$
\end{lemma}
\begin{proof}
    Let $u^\eps$ be the elliptic regularization of the arrival time,
and write
\[
    \Sigma_s^\eps:=\{u^\eps=s\},\qquad
    d\mu_s^\eps:=d\mathcal H^n\llcorner \Sigma_s^\eps .
\]
Then we can apply the approximation argument in the proof of Proposition \ref{thm:ae-minkowski}, replace relaxed Minkowski inequality by the Michael--Simon Sobolev inequality \cite{MS73, brendle2021isoperimetric} to smooth hypersurface ${\Sigma_s^{\varepsilon}}$
\begin{equation}\label{eq:MS-inequality}
\int_{\Sigma_s^{\varepsilon}}
|H_{\Sigma_s^{\varepsilon}}|\,d\mathcal H^n\geq c_{\mathrm{MS}}(n)|{\Sigma_s^{\varepsilon}}|^{\frac{n-1}{n}}
\end{equation}
and use Lemma \ref{strong slice conv} to conclude the proof of Lemma \ref{lem:MS-inequality}.
\end{proof}

\section{Mean convex approximation}\label{mean convex approximation}

To prove the optimal extinction time estimates for weak mean curvature flow, we will take the outward minimizing hull $\Omega^*$ of smooth bounded domain $\Omega$. $\partial \Omega^*$  has at most  $C^{1, 1}$ regularity in the contact part and has codimension $7$ singularities in the free minimal part. Hence, we need to first construct a smooth, mean convex envelope of an arbitrary domain $\Omega$ with smooth boundary.
We will show that the level set flow starting from these envelopes will have extinction time
whose limit supremum is less than or equal to the right hand side of \eqref{eq:general-extinction}.
Then, maximum principle gives the desired upper bound
for the extinction time of the level set flow starting from $\Omega$.

The construction involves three steps.
First, $\Omega^*$ admits a smooth but not necessarily mean convex approximation $U_i$,
by Schmidt's interior approximation theorem \cite{Schmidt2015}.
Then, we consider a penalized area functional for finite perimeter sets containing $U_i$,
whose minimizer satisfies a prescribed mean curvature equation.
Consequently, the minimizer turns out to be $C^\infty$--quasi--regular
having strictly positive mean curvature on its regular part.
This allows us to apply Bi--Zhu's mean--convex approximation \cite{BZ26CurvatureFree},
which was used in their recent paper to prove the Riemannian Penrose inequality
for general dimensions \cite{BZ26Penrose}. The precise statement of the approximation lemma is as follows:

\begin{lemma}[Smooth strictly mean convex approximation of an outward minimizing hull]\label{lem:hull-approximation}
Let $\Omega \Subset \mathbb{R}^{n+1}$ be a bounded open set with smooth
boundary, and let $\Omega^*$
be its strictly outward-minimizing hull. Then there exists a large ball $B$ and
a sequence of bounded smooth open sets $\{\Omega_i\}_{i=1}^\infty$ such that
\[
\Omega^* \Subset \Omega_i \subset B,
\qquad
H_{\partial\Omega_i}>0,
\]
where the mean curvature is computed with respect to the outward unit
normal, and
\begin{equation}\label{eq:area-approximation} 
P(\Omega_i)\leq P(\Omega^*)+i^{-1}.
\end{equation}
Consequently, for every finite-perimeter set $F\supset \Omega_i$,
\begin{equation}\label{eq:initial-almost-outward-minimizing}
P(\Omega_i)\leq P(F)+i^{-1}.
\end{equation}
\end{lemma}

\begin{proof}
Without loss of generality, suppose $\Omega$ contains the origin.
By \cite[Lemma~5.5]{AFM22} (see also \cite{Schmidt2015}), there are bounded smooth
open sets $U_i$ such that
\begin{equation}\label{eq:Ui-perimeter}
\Omega^* \subset U_i,
\quad
P(U_i)\leq P(\Omega^*)+(4i)^{-1},
\quad \text{and } \dist_{\cH}(\partial U_i,\partial\Omega^*)\leq (4i)^{-1},
\end{equation}
where $\dist_\cH$ denotes the Hausdorff distance.
Let $B_R$ be the ball of radius $R$ centered at the origin,
such that $U_i\Subset B_{R/2}$ for all $i$.
Consider the problem of minimizing the penalized area functional
\[
\mathcal J_i(F)
:=
P(F)-\int_{F\setminus U_i}q_i\,dx
\]
on the class of finite perimeter sets $F\supset U_i$ of finite measure, where
\[
q_i\in C^\infty(\mathbb{R}^{n+1}\setminus U_i)
\]
satisfies
\begin{align}
q_i&>0
    &&\text{on }\mathbb{R}^{n+1}\setminus U_i,
    \label{eq:qi-positive}\\
q_i&>H_{\partial U_i}
    &&\text{on }\partial U_i,
    \label{eq:qi-barrier}\\
\int_{\mathbb{R}^{n+1}\setminus U_i}q_i\,dx
    &<(4i)^{-1}.
    \label{eq:qi-integral}
\end{align}
Moreover, we require $q_i$ to decay sufficiently rapidly that,
\begin{equation}\label{eq:qi-radial}
q_i(x)<\frac{n}{2|x|}
\quad\text{for }|x|\geq R.
\end{equation}
Such a function can be constructed as follows. Choose
$
a_i>\sup H_{\partial U_i}
$
and a nonnegative smooth cutoff $\psi_i$ which equals one on
$\partial U_i$ and is supported in a sufficiently thin tubular
neighborhood of $\partial U_i$. By making the tubular neighborhood thin,
we have
\[
a_i\int_{\mathbb{R}^{n+1}\setminus U_i}\psi_i\,dx<(8i)^{-1}.
\]
We force $\supp \psi_i \Subset B_R$ and choose $\varepsilon_i>0$ sufficiently small. Setting
\[
q_i(x):=a_i\psi_i(x)+\varepsilon_i e^{-\eps_i^{-1}|x|^2},
\]
\eqref{eq:qi-positive}--\eqref{eq:qi-radial} follow.

We first show that the minimization problem for $\cJ_i$ admits a minimizer $F_i$ that satisfy $U_i \subset F_i \subset B_{R}$. For a calibration-type argument, we consider the unit radial vector field
\[Y_i(x)=\frac{x}{|x|},\]
which is a smooth vector field on $\R^{n+1}\setminus \{0\}$ satisfying
\[
|Y_i|=1, \quad Y_i=\nu_{B_r}\quad\text{for }r>0,
\quad
\operatorname{div}Y_i=\frac{n}{|x|}.
\]
For a finite perimeter set $F$ containing $U_i$, a standard inequality for perimeter gives
\[
P(F\cap B_{R})+P(F\cup B_{R})
\leq P(F)+P(B_{R}).
\]
Applying the Gauss-Green formula, we obtain
\[
P(B_{R})-P(B_{R/2})=\int_{B_{R}\setminus B_{R/2}}\operatorname{div}Y_i\,dx
\]
and
\[
P(F\cup B_{R})-P(B_{R/2})
\geq
\int_{(F\cup B_{R})\setminus B_{R/2}}\operatorname{div}Y_i\,dx.
\]
Consequently,
\begin{align*}
P(F\cap B_{R})-P(F)
&\leq P(B_{R})-P(F\cup B_{R})\\
&\leq
-\int_{F\setminus B_{R}}\operatorname{div}Y_i\,dx\\
&=
-\int_{F\setminus B_{R}}\frac{n}{|x|}\,dx.
\end{align*}
Since $U_i\subset B_{R}$, it follows from \eqref{eq:qi-radial} that
\begin{align}
\mathcal J_i(F\cap B_{R})-\mathcal J_i(F)
&\leq
-\int_{F\setminus B_{R}}
\left(\frac{n}{|x|}-q_i(x)\right)\,dx
\leq 0.
\label{eq:radial-truncation}
\end{align}
Therefore, a minimizing sequence of $\cJ_i$ may be chosen inside $B_{R}$. Since $B_{R}$ is bounded, this yields a minimizer $F_i$ of the original admissible class, which satisfies $U_i\subset F_i \subset B_{R}$. Here, without changing the value of $\cJ_i(F_i)$, we can replace $F_i$ with its measure-theoretic interior
\[F_i^{(1)}=\left\{x\in\R^{n+1}: \lim_{r\rightarrow 0+}\frac{|B_r(x)\cap F_i|}{|B_r(x)|}=1\right\}.\]
This is because $F_i^{(1)}$ and $F_i$ agree up to $\mathcal{L}^{n+1}$-null set by Lebesgue differentiation theorem. A convenient property of $F_i^{(1)}$ is that
\[\partial F_i^{(1)}=\overline{\partial^*F_i}.\]
That is, the topological boundary of $F_i^{(1)}$ agrees with the closure of the reduced boundary of $F_i$. Moreover, $\partial F_i^{(1)}$ agrees with
\[\operatorname{spt}|D\chi_{F_i}|=\{x\in\R^{n+1}: |B_r(x)\cap F_i|, |B_r(x)\cap F_i^c|>0 \; \text{for} \; r>0\}.\]
For a proof, see \cite[Theorem 10]{Caraballo2011} and \cite[Proposition 12.19]{Mag12}.
From now on, we will work with this representative and write $F_i$ for $F_i^{(1)}$. A comparison with $U_i$ gives
\[
P(F_i)-\int_{F_i\setminus U_i}q_i\,dx
\leq P(U_i),
\]
and hence, by \eqref{eq:qi-integral},
\begin{equation}\label{eq:Fi-perimeter}
P(F_i)\leq P(U_i)+(4i)^{-1}.
\end{equation}

We next show that the minimizer separates from the  inner obstacle $U_i$.
Let $d_i$ be the signed distance from $\partial U_i$, chosen positive on
$\mathbb{R}^{n+1}\setminus\overline{U_i}$. Since $\partial U_i$ is smooth,
$d_i$ is smooth in a tubular neighborhood of $\partial U_i$, and
\[
\nabla d_i=\nu_{U_i},
\qquad
\operatorname{div}(\nabla d_i)=H_{\partial U_i}
\quad\text{on }\partial U_i.
\]
By \eqref{eq:qi-barrier} and continuity, after shrinking the tubular neighborhood, there
exist $\rho_i>0$ and $c_i>0$ such that $d_i$ is smooth in the collar \[C_i:=\{-2\rho_i<d_i<2\rho_i\},\]
\[|\nabla d_i|>0 \quad \text{on } C_i \]
so that each sublevel set $\{d_i<s\}$ for $-2\rho_i<s<2\rho_i$ has smooth boundary, and
\begin{equation}\label{eq:strict-collar}
q_i-\operatorname{div}(\nabla d_i)\geq c_i
\quad\text{on }\{0\leq d_i< 2\rho_i\}.
\end{equation}

Writing $X_i=\nabla d_i$, this is a smooth vector field on $C_i$ satisfying $|X_i|\leq 1$. For $0<s\leq\rho_i$, set
\[
U_i^s:=\{d_i<s\}.
\]
Again, we have
\[
P(F_i\cup U_i^s)-P(F_i)
\leq
P(U_i^s)-P(F_i\cap U_i^s).
\]
Since $X_i=\nu_{U_i^s}$ on $\partial U_i^s$, the Gauss--Green formula
and $|X_i|\leq1$ give
\[
P(U_i^s)-P(U_i)
=
\int_{U_i^s\setminus U_i}\operatorname{div}X_i\,dx
\]
and
\[
P(F_i\cap U_i^s)-P(U_i)
\geq
\int_{(F_i\cap U_i^s)\setminus U_i}\operatorname{div}X_i\,dx.
\]
Therefore
\begin{equation}\label{eq:collar-perimeter}
P(F_i\cup U_i^s)-P(F_i)
\leq
\int_{U_i^s\setminus F_i}\operatorname{div}X_i\,dx.
\end{equation}
Using \eqref{eq:strict-collar} and the minimality of
$F_i$, we find
\begin{align*}
0
&\leq
\mathcal J_i(F_i\cup U_i^s)-\mathcal J_i(F_i)\\
&\leq
\int_{U_i^s\setminus F_i}
\bigl(\operatorname{div}X_i-q_i\bigr)\,dx\\
&\leq -c_i|U_i^s\setminus F_i|.
\end{align*}
It follows that
\[
|U_i^s\setminus F_i|=0
\qquad\text{for every }0<s\leq\rho_i,
\]
and in particular $\operatorname{dist}(\partial  F_i,\partial U_i)>0$.

We now prove that $\partial F_i$ is almost minimizing. Choose
\[
0<r_i^0<
\frac14\operatorname{dist}(\partial F_i,\partial U_i).
\]
Let $x\in\partial F_i$, let $r<r_i^0$, and suppose that $G$ is a
finite-perimeter set satisfying
\[
F_i\triangle G\Subset B_r(x).
\]
Since $B_r(x)\cap U_i=\varnothing$, the set $G$ still contains $U_i$ and is
therefore admissible. Minimality gives
\begin{align*}
P(F_i;B_r(x))
&\leq
P(G;B_r(x))
+\int_{F_i\setminus G}q_i\,dx
-\int_{G\setminus F_i}q_i\,dx\\
&\leq
P(G;B_r(x))
+\|q_i\|_{L^\infty(\mathbb{R}^{n+1})}
\,|F_i\triangle G|\\
&\leq
P(G;B_r(x))
+\Lambda_i r^{n+1},
\end{align*}
where
\[
\Lambda_i
:=
\omega_{n+1}\|q_i\|_{L^\infty(\mathbb{R}^{n+1})}.
\]
Thus a variation of $F_i$ within $B_r(x)$ only results in a perimeter gain of $O(r^{n+1})$, which decays faster than $r^n$ as $r\rightarrow 0$.
The regularity theorem of Almgren--Bombieri--Tamanini \cite{Alm75, Bom82, Tam84} for almost-minimizing boundaries gives a decomposition
\[
\partial F_i
=
\operatorname{Reg}(\partial F_i)
\sqcup
\operatorname{Sing}(\partial F_i),
\]
where $\operatorname{Reg}(\partial F_i)$ is relatively open and
$C^{1,\alpha_i}$ for some $\alpha_i\in(0,1)$, and
\[
\dim_{\mathcal H}\operatorname{Sing}(\partial F_i)
\leq (n+1)-8=n-7.
\]

Since $\partial F_i$ is separated from $U_i$, all compactly supported
variations near $\partial F_i$ are admissible in both directions. Let
$X\in C_c^\infty(\mathbb{R}^{n+1}\setminus\overline{U_i};\mathbb{R}^{n+1})$,
and let $\Phi_t$ be its flow. Computing the first variation of $\cJ_i$ for
$\Phi_t(F_i)$ at $t=0$ gives
\begin{equation}\label{eq:first-variation-Fi}
\int_{\partial F_i}
\left(
\operatorname{div}_{\partial F_i}X
-q_i\,X\cdot\nu_i
\right)
\,d\mathcal H^n
=0.
\end{equation}
Consequently, on the regular part,
\begin{equation}\label{eq:H-equals-q}
H_{\partial F_i}=q_i
\end{equation}
with respect to the outward unit normal.

Locally, $\operatorname{Reg}(\partial F_i)$ is a $C^{1,\alpha_i}$ graph,
and \eqref{eq:H-equals-q} is the prescribed-mean-curvature equation with
smooth right-hand side. Standard quasilinear elliptic regularity first gives
$C^{2,\alpha_i}$ regularity and then, by bootstrapping,
\[
\operatorname{Reg}(\partial F_i)
\quad\text{is smooth}.
\]

As noted in \cite[Remark 3.10]{BZ26Penrose}, together with the almost minimizing property, this gives quasi--regularity of $\operatorname{Int} F_i$ (For definitions, see \cite[Section 3]{BZ26CurvatureFree}).
Indeed, first note that almost-minimizing boundary admits a tangent minimal cone at each point of the boundary \cite{Tam84}.
If $p\in\partial F_i$ is a point at which some ball in $F_i$
can be inscribed at, the tangent cone at $p$ is contained in a closed half-space.
By the maximum principle, this forces the tangent cone at $p$ to be a plane \cite{Sim87}, so $p\in\text{Reg}(\partial F_i)$.
Conversely, if $p\in\text{Reg}(\partial F_i)$, by taking a smooth boundary chart we can find a ball in $F_i$ inscribed at $p$.

The set $\partial F_i$ is compact, since $F_i\subset B_{R_i}$, and it is
disjoint from $\overline{U_i}$. Since $q_i>0$ on
$\mathbb{R}^{n+1}\setminus U_i$, we have
\[
\gamma_i
:=
\min_{\partial F_i}q_i
>0.
\]
It follows from \eqref{eq:H-equals-q} that
\begin{equation}\label{eq:positive-H-Fi}
H_{\partial F_i}\geq\gamma_i>0
\quad\text{on }\operatorname{Reg}(\partial F_i).
\end{equation}

We may now apply the mean convex smoothing theorem \cite[Proposition~2.6]{BZ26CurvatureFree} to $\operatorname{Int}F_i$ in the smooth exterior region
$\mathbb{R}^{n+1}\setminus\overline{U_i}$. It gives smooth compact
hypersurfaces $\Gamma_{i,k}$ approximating $\partial F_i$ from the interior
side such that
\[
H_{\Gamma_{i,k}}\geq\gamma_i-k^{-1}
\]
with respect to the outward normal,
\begin{equation}\label{eq:hausdorff-conv}
\operatorname{dist}_\cH(\Gamma_{i,k},\partial F_i)<k^{-1},
\end{equation}
and
\[
\mathcal H^n(\Gamma_{i,k})
\rightarrow
\mathcal H^n(\partial F_i)
=
P(F_i).
\]
Let $\Omega_{i,k}$ be the bounded open set enclosed by $\Gamma_{i,k}$.
For all sufficiently large $k$,
\[
U_i\Subset\Omega_{i,k}.
\]
Choose $k=k(i)$ sufficiently large that this inclusion, mean convexity of $\partial\Omega_i$, and
\[
P(\Omega_{i,k(i)})
\leq P(F_i)+(4i)^{-1}
\]
holds.
Set
\[
\Omega_i:=\Omega_{i,k(i)},
\]
we claim that $\Omega_i$ are the desired approximations with smooth boundary. Indeed, we first have
\[
\Omega^*\subset U_i\Subset\Omega_i\subset B_{R+1},
\quad \text{and} \quad
H_{\partial\Omega_i}>0,
\]
where the last inclusion holds by \eqref{eq:hausdorff-conv}.
Combining \eqref{eq:Ui-perimeter}, \eqref{eq:Fi-perimeter}, and the choice
of $\Omega_i$, 
\[
P(\Omega_i)\le P(F_i)+\frac1{4i}
\le P(U_i)+\frac1{2i}
\le P(\Omega^*)+\frac{3}{4i}
<P(\Omega^*)+\frac1i.
\]

Finally, let $F$ be any finite-perimeter set satisfying $F\supset\Omega_i$.
Since $\Omega^*$ is outward minimizing and $
F\supset\Omega_i\supset \Omega^*,
$
\[
P(\Omega^*)\leq P(F).
\]
Therefore
\[
P(\Omega_i)
\leq P(\Omega^*)+i^{-1}
\leq P(F)+i^{-1}.
\]
This completes the proof.
\end{proof}

Given the approximations $\Omega_i$, let
\[
K_{i,s}:=\{u_i\geq s\},
\qquad
\Sigma_{i,s}:=\partial^{*}K_{i,s},
\qquad
\mu_{i,s}:=\mathcal H^n\llcorner\Sigma_{i,s},
\]
where $u_i:\Omega_i\rightarrow[0,T_i]$ is the arrival-time function for mean convex level-set flow,
and define the area function
\[
A_i(s):=\mu_{i,s}(\mathbb R^{n+1})=\mathcal{H}^n(\Sigma_{i,s}).
\]

From the first variation of perimeter
\begin{equation}
A_i'(s)
=
-\int_{\Sigma_{i,s}}H_i^2\,d\mu_{i,s},
\end{equation}
Proposition \ref{thm:ae-minkowski} and Lemma \ref{lem:MS-inequality} yield a differential inequality for $A_i(s)$.
Integrating the  differential inequality to the extinction time
\[T_i=\inf\{t\geq0: K_i(t)=\varnothing\}=\inf\{t\geq0: A_i(t)=0\},\]
we obtain the extinction time estimate for the approximating mean convex flows.

\begin{proposition}[Extinction time of the approximating flows]
\label{prop:approximate-extinction}
For all sufficiently large \(i\),
\begin{equation}\label{eq:Ti-estimate}
T_i
\leq
\mathcal{T}_i:=\frac{1}{2n}
\left(
\frac{A_i(0)}{|{\mathbb{S}^n}|}
\right)^{{\frac{2}{n}}}
\left(1-i^{-1/2}\right)^{-\frac{2(n-1)}{n}}
+
\frac{n}{2c_{\mathrm{MS}}(n)^2}i^{-1/n}.
\end{equation}
In particular, we have
\begin{equation}\label{eq:limsup-Ti}
\limsup_{i\to\infty}T_i
\leq
\frac{1}{2n}
\left(
\frac{P(\Omega^*)}{|{\mathbb{S}^n}|}
\right)^{{\frac{2}{n}}}.
\end{equation}
\end{proposition}
\begin{proof}
By Lemma \ref{lem:no-mass-drop}, $A_i$ is absolutely continuous, with
\begin{equation}\label{eq:area-equality-i}
A_i'(s)
=
-\int_{\Sigma_{i,s}}H_i^2\,d\mu_{i,s}
\end{equation}
for almost every \(s\in(0,T_i)\).
For each $i$, put \[\tau_i
:=
\inf\left\{
s\in[0,T_i]:
A_i(s)\leq i^{-1/2}
\right\}.\]
Since \(A_i(T_i)=0\), $\tau_i\in[0,T_i]$, and so by the continuity of \(A_i:[0,T_i]\rightarrow\mathbb{R}_{\geq0}\), we have
\[
A_i(\tau_i)\leq i^{-1/2}.
\]

For \(s<\tau_i\), we have \(A_i(s)>i^{-1/2}\), and hence
\begin{equation}\label{before zero}
A_i(s)-\frac{1}{i}>A_i(s)-i^{-1/2}A_i(s)
=\left(1-i^{-1/2}\right)A_i(s).
\end{equation}
Combining \eqref{eq:area-equality-i}, Jensen's inequality,
Proposition \ref{thm:ae-minkowski}, and \eqref{before zero}, we obtain
\begin{align*}
A_i'(s)
&\leq
-A_i(s)^{-1}
\left(
\int_{\Sigma_{i,s}}H_i\,d\mu_{i,s}
\right)^2
\\
&\leq
-A_i(s)^{-1}n^2|{\mathbb{S}^n}|^{{\frac{2}{n}}}
\left(A_{i}(s)-i^{-1}\right)^{\frac{2(n-1)}{n}}
\\
&\leq
-n^2|{\mathbb{S}^n}|^{{\frac{2}{n}}}
\left(1-i^{-1/2}\right)^{\frac{2(n-1)}{n}}
A_i(s)^{\frac{n-2}{n}}
\end{align*}
for a.e. $s\in[0,\tau_i)$. Therefore,
\begin{equation}\label{eq:area-odi}
\frac{d}{ds}A_i(s)^{{\frac{2}{n}}}
\leq
-2n|{\mathbb{S}^n}|^{{\frac{2}{n}}}
\left(1-i^{-1/2}\right)^{\frac{2(n-1)}{n}}
\end{equation}
for a.e. $s\in[0,\tau_i)$. Integrating from \(0\) to
\(\tau_i\), we obtain
\begin{equation}\label{eq:tau-i-bound}
\tau_i
\leq
\frac{1}{2n}
\left(
\frac{A_i(0)}{|{\mathbb{S}^n}|}
\right)^{{\frac{2}{n}}}
\left(1-i^{-1/2}\right)^{-\frac{2(n-1)}{n}}.
\end{equation}

It remains to control \(T_i-\tau_i\). Combining
\eqref{eq:area-equality-i}, Jensen's inequality, and
\eqref{eq:MS-inequality} in turn, we obtain
\begin{align*}
A_i'(s)
&\leq
-A_i(s)^{-1}
\left(
\int_{\Sigma_{i,s}}H_i\,d\mu_{i,s}
\right)^2
\\
&\leq
-A_i(s)^{-1}(c_{\mathrm{MS}}(n)
A_i(s)^{\frac{(n-1)}{n}})^2
\\
&\leq
-c_{\mathrm{MS}}(n)^2
A_i(s)^{\frac{n-2}{n}}
\end{align*}
for almost every \(s\in(\tau_i,T_i)\). Hence
\begin{equation}\label{eq:area-odi2}
\frac{d}{ds}A_i(s)^{{\frac{2}{n}}}
\leq
-\frac{2}{n}c_{\mathrm{MS}}(n)^2.
\end{equation}
Integrating up to any \(t<T_i\) and then letting \(t\rightarrow T_i\),
we obtain
\begin{equation}\label{eq:tail-bound}
T_i-\tau_i\leq
\frac{n}{2c_{\mathrm{MS}}(n)^2}A_i(\tau_i)^{{\frac{2}{n}}}\leq \frac{n}{2c_{\mathrm{MS}}(n)^2}i^{-1/n}.
\end{equation}

Combining \eqref{eq:tau-i-bound} and \eqref{eq:tail-bound} proves
\eqref{eq:Ti-estimate}.
Finally, the outward-minimizing property of $\Omega^*$ and
\eqref{eq:area-approximation} give
\[
P(\Omega^*)
\leq
P(\Omega_i)
\leq
P(\Omega^*)+\frac1i.
\]
Thus
\[
A_i(0)=\mathcal{H}^n(\{u_i=0\})=P(\Omega_i)\rightarrow P(\Omega^*),
\]
and so \eqref{eq:limsup-Ti} follows from (\ref{eq:Ti-estimate}).
\end{proof}

\section{Proof of the main results}\label{proofofmainresults}

In this section, we prove the main results about optimal extinction time estimates and optimal parabolic isoperimetric estimates for weak mean curvature flow in Theorem \ref{Optimal extinction time estimates for weak mean curvature flow}, Theorem \ref{lpestimates arrival time} and Theorem \ref{sharp parabolic estimates}.  To this end, we need the following proposition.

\begin{proposition}\label{thm:LSF-Brakke extinction optimal}
Let $\Omega\subset\R^{n+1}$  be a bounded domain with smooth boundary,
and let $\{K_t\}_{t\geq 0}$ be the level set flow starting from $\Omega$ and  $\{\mu_t\}_{t\geq 0}$ be an integral Brakke flow starting from the associated Radon measure $\mu_0=\mathcal{H}^n\llcorner\partial\Omega$,  for extinction times $T_{\mathrm{ext}}(\Omega)
:=
\inf\{t\geq0:K_t=\varnothing\}$ and $T^{B}_{\mathrm{ext}}(\Omega)
:=
\inf\{t\geq0:\mu_t\equiv0\},$
we have
\begin{equation}\label{eq:general-extinction6}
T_{\mathrm{ext}}(\Omega)
\leq
\frac{1}{2n}
\left(
\frac{P(\Omega^*)}{|{\mathbb{S}^n}|}
\right)^{\frac{2}{n}}
\end{equation}
and
\begin{equation}\label{eq:general-brakke-extinction6}
T^{B}_{\mathrm{ext}}(\Omega)
\leq
\frac{1}{2n}
\left(
\frac{P(\Omega^*)}{|{\mathbb{S}^n}|}
\right)^{\frac{2}{n}}
\end{equation}
where $P(\Omega^*)$ is the perimeter of the outward minimizing hull $\Omega^*$ of $\Omega$.
In addition, the equality in \eqref{eq:general-extinction6} holds,
if and only if $\Omega^*$ is a round ball. In particular, if $\Omega$ has only one smooth boundary component, then equality in \eqref{eq:general-extinction6} holds
if and only if $\Omega=\Omega^*$ is a round ball. Moreover, if $\Omega$ has only one smooth boundary component, then equality in \eqref{eq:general-brakke-extinction6} holds, if and only if $\Omega=\Omega^*$ is a round ball and the Brakke flow $\{\mu_t\}$ is the multiplicity-one standard  smoothly shrinking round sphere Brakke flow.
\end{proposition}

\begin{proof}[Proof of Proposition  \ref{thm:LSF-Brakke extinction optimal}]
We apply Lemma \ref{lem:hull-approximation} to $\Omega^*$, to obtain the smooth approximations $\Omega_i$ whose extinction time $T_i$ satisfies Proposition \ref{prop:approximate-extinction}.
Let \(K_{i,t}\) be the closed region enclosed by the mean convex level-set flow starting from \(\Omega_i\) at time $t$. Since $\Omega^*
\subset
\Omega_i$, we have
\begin{equation}\label{eq:inclusion}
K_t \subset K_{i,t}
\end{equation}
for every \(t\geq0\).

For every \(t>T_i\), the level-set flow starting from \(\Omega_i\)
is empty, so \(K_{i,t}=\varnothing\). It follows from
\eqref{eq:inclusion} that 
$K_t=\varnothing$ for every $t>T_i$. Consequently, 
$T_{\mathrm{ext}}(\Omega)\leq T_i$ for every \(i\), and hence by (\ref{eq:limsup-Ti}) we get
\begin{equation}\label{TTilimit}
    T_{\mathrm{ext}}(\Omega)\leq \limsup_{i\to\infty}T_i\leq\frac{1}{2n}
\left(
\frac{P(\Omega^*)}{|{\mathbb{S}^n}|}
\right)^{{\frac{2}{n}}}.
\end{equation}
This proves \eqref{eq:general-extinction6}.

Now we discuss the equality case. %The rough idea is  deriving the convergence to spherical area dispassion  and  using it to show the mean convergence to constant mean curvature, and then applying  Alexandrov’s theorem. However, we have to handle convergence issue to make this idea rigorous. 
We first consider the equality case \eqref{eq:general-extinction6} of level set flow. If the outward minimizing hull $\Omega^*$ is a round ball which has mean convex boundary, the free part $\partial \Omega^*\setminus \partial \Omega$ is empty. $\Omega$ has a boundary component that coincides with the sphere $\partial\Omega^*$. Since $\Omega\subset \Omega^*$ The boundary component other than the sphere $\partial\Omega^*$ lives in the interior $\textrm{Int}(\Omega^*)$ of the ball $\Omega^*$. By avoidance principle and direct computation on the extinction time of $\Omega^*$, $\Omega$ achieves the equality in \eqref{eq:general-extinction6}. 

Then, suppose  the equality case of \eqref{eq:general-extinction6} holds, by \eqref{eq:inclusion} we have
\[T_i\geq T_{\rm ext}=\frac{1}{2n} \left(\frac{P(\Omega^*)}{|{\mathbb{S}^n}|}\right)^{{\frac{2}{n}}}\]
for every $i$.

Suppose we are given $t\in[0,T_\textrm{ext})$.
Following the same notations of the previous proposition,
by \eqref{eq:tail-bound} we have $t\leq \tau_i$ for all sufficiently large $i$.
For such $i$, from the differential inequality \eqref{eq:area-odi} we have
\begin{equation}\label{eq:area-squeeze}
A_i(\tau_i)^{{\frac{2}{n}}}+c_i(\tau_i-t)\leq A_i(t)^{{\frac{2}{n}}}\leq A_i(0)^{{\frac{2}{n}}}-c_i t,
\end{equation}
where \[c_i=2n|{\mathbb{S}^n}|^{{\frac{2}{n}}}\left(1-i^{-1/2}\right)^{\frac{2(n-1)}{n}}.\]
We have $\tau_i\to T_{\textrm{ext}}$, because $T_i-\tau_i\to 0$ by \eqref{eq:tail-bound}
and $T_i \to T_{\textrm{ext}}$ by \eqref{eq:Ti-estimate}.
Along with the fact that $A_i(0)\to P(\Omega^*)$
and $A_i(\tau_i)= i^{-1/2}$, we can take the limit as $i\to\infty$ in \eqref{eq:area-squeeze} to obtain
\[A_i(t)^{{\frac{2}{n}}}\to P(\Omega^*)^{{\frac{2}{n}}}-c t,\quad \text{with } c=2n|{\mathbb{S}^n}|^{{\frac{2}{n}}}.\]
If $R>0$ is the radius of a ball with perimeter $P(\Omega^*)$, i.e. $P(\Omega^*)=|{\mathbb{S}^n}|R^n$,
then we have the following convergence of $A_i(t)$ to the area of a shrinking sphere at time $t$:
\begin{equation}\label{eq:Ai-convergence}
A_i(t) \to A(t)\coloneq|{\mathbb{S}^n}|(R^2-2nt)^{n/2}.
\end{equation}
Using this information, we claim that for a.e. $t<T_{\textrm{ext}}$,
\begin{equation}\label{Hi-hintegral2to0}
\int_0^{t} \int_{\Sigma_{i,s}} (H_i-h(s))^2\,d\mu_{i,s}\,ds \to 0\quad \text{as }i\to\infty,
\end{equation}
where $h(s)$ is the mean curvature of a shrinking sphere at time $s$:
\[h(s)=\frac{n}{\sqrt{R^2-2ns}}.\]
Indeed, expanding the square and applying Lemma \ref{lem:no-mass-drop},
\begin{align*}
&\quad\int_0^{t} \int_{\Sigma_{i,s}} (H_i-h(s))^2\,d\mu_{i,s}\,ds\\
&=A_i(0)-A_i(t)-2\int_0^{t} h(s)\int_{\Sigma_{i,s}} H_i\,d\mu_{i,s}\,ds+\int_0^{t} h(s)^2A_i(s)\,ds.
\end{align*}
Here, recall the inequality from Proposition \ref{thm:ae-minkowski}, which holds for almost every $s\in(0,t)$:
\[\int_{\Sigma_{i,s}} H_i\,d\mu_{i,s}\geq n|{\mathbb{S}^n}|^{1/n}(A_i(s)-i^{-1/2})^{(n-1)/n}.\]
Plugging above and sending $i\to\infty$, by \eqref{eq:Ai-convergence} we obtain
\begin{align*}
\limsup_{i\to\infty}\int_0^{t} \int_{\Sigma_{i,s}} &(H_i-h(s))^2\,d\mu_{i,s}\,ds\leq A(0)-A(t)-\int_0^t h(s)^2 A(s) ds = 0,
\end{align*}
proving the claim.

Now, fix a sequence $\eta_i\to0$ of times that satisfy \eqref{Hi-hintegral2to0}. By taking $\delta_i\to0$ with sufficiently slow convergence so that
\[\int_0^{\eta_i} \int_{\Sigma_{i,s}} (H_i-h(s))^2\,d\mu_{i,s}\,ds\ll\delta_i,\]
by averaging we can find times $0<t_i<\delta_i$  such that 
\begin{equation}\label{eq:mean-curvature-convergence}
\int_{\Sigma_{i,t_i}} (H_i-h(t_i))^2 \,d\mu_{i,t_i}\leq2\delta_i^{-1}\int_0^{\delta_i} \int_{\Sigma_{i,s}} (H_i-h(s))^2\,d\mu_{i,s}\,ds\to 0.
\end{equation}
We aim to show that
\begin{equation}\label{eq:chi-convergence}
    \chi_{K_{i,t_i}}\to \chi_{\Omega^*}\quad\text{in }L^1(\mathbb R^{n+1}),
\end{equation}
and pass \eqref{eq:mean-curvature-convergence} to the limit
to show that $\partial^*\Omega^*$ has constant mean curvature.
First, recall that
\[
  \Omega^*\subset\Omega_i,
  \qquad
  P(\Omega^*)\le P(\Omega_i)\le P(\Omega^*)+i^{-1}.
\]
By local $BV$ compactness and lower
semicontinuity of perimeter, every subsequence of $(\Omega_i)$ has a further
subsequence (which we still denote by $(\Omega_i)$) such that $\chi_{\Omega_i}\to\chi_F$ in
$L^1_{\mathrm{loc}}$, with
\[
  \Omega^*\subset F\quad\text{up to a null set},
  \qquad
  P(F)\le\liminf_{i\to\infty}P(\Omega_i)=P(\Omega^*).
\]
Since $\Omega^*$ is strictly outward minimizing, $F=\Omega^*$ up to a null set.
Note that there exists a large ball containing all $\Omega_i$.
Consequently, the whole sequence converges:
\begin{equation}\label{eq:chi-omega-convergence}
  \chi_{\Omega_i}\to\chi_{\Omega^*}
  \quad\text{in }L^1(\mathbb R^{n+1}).
\end{equation}

Next, we claim that $|\Omega_i\setminus K_{i,t_i}|\to0$.
As $\Omega_i$ is mean convex, the associated level set function $u_i$ satisfies
$|Du_i|>0$ almost everywhere and $H_i=|Du_i|^{-1}$ on almost
every level set; see
\cite[Lemma 2.3, Theorem 1.3, and Lemma 4.4]{MS08}.
The co-area formula therefore gives
\[
  |\Omega_i\setminus K_{i,t_i}|
  =\int_0^{t_i}\int_{\Sigma_{i,s}}
       H_i\,d\mu_{i,s}\,ds.
\]
Applying Cauchy--Schwarz inequality and Lemma \ref{lem:no-mass-drop}, we obtain
\begin{align*}
  |\Omega_i\setminus K_{i,t_i}|
  &\le
  \left(\int_0^{t_i}A_i(s)\,ds\right)^{1/2}
  \left(\int_0^{t_i}\int_{\Sigma_{i,s}}
             H_i^2\,d\mu_{i,s}\,ds\right)^{1/2}\\
  &\le
  \bigl(t_iA_i(0)\bigr)^{1/2}
  \bigl(A_i(0)-A_i(t_i)\bigr)^{1/2}\\
  &\le A_i(0)\sqrt{t_i}\to0.
\end{align*}
Since $K_{i,t_i}\subset {\Omega_i}$, this convergence together with \eqref{eq:chi-omega-convergence}
 implies the desired convergence \eqref{eq:chi-convergence}.
This yields a useful corollary for the area of the level sets at time $t_i$.
By \eqref{eq:chi-convergence}, lower semicontinuity of perimeter, and the monotonicity of perimeter under mean convex mean curvature flow, we have
\[
  P(\Omega^*)
  \le\liminf_{i\to\infty}A_i(t_i)
  \le\limsup_{i\to\infty}A_i(t_i)
  \le\lim_{i\to\infty}A_i(0)=P(\Omega^*),
\]
thus
\begin{equation}\label{eq:aiti-convergence}
    A_i(t_i)\to P(\Omega^*).
\end{equation}

Write $\lambda_i$ and $\lambda$ for the Gauss-Green measures of
$K_{i,t_i}$ and $\Omega^*$, respectively. Namely, $\lambda_i := \nu_i\,\mathcal{H}^n\llcorner \partial^*K_{i,t_i},
\lambda := \nu\,\mathcal{H}^n\llcorner \partial^*\Omega^*$ and $\nu_i, \nu$ are the corresponding measure-theoretic outward unit normal.

For every vector field $X\in C_c^1(\mathbb R^{n+1};\mathbb R^{n+1})$, we have
\[
  \int X\cdot d(\lambda_i-\lambda)
  =\int_{\mathbb R^{n+1}}
       (\chi_{K_{i,t_i}}-\chi_{\Omega^*})
       \operatorname{div}X\,dx
  \to0.
\]
Noting that
\[
  |\lambda_i|(\mathbb R^{n+1})=A_i(t_i)
  \to P(\Omega^*)
  =|\lambda|(\mathbb R^{n+1}),
\]
the total variation of $\lambda_i$ is uniformly bounded.
By uniform approximation, for $X\in C_0$ we have
\[\int X\cdot d\lambda_i\to\int X\cdot d\lambda,\]
hence $\lambda_i\rightharpoonup\lambda$ as vector-valued Radon measures.
Together with \eqref{eq:aiti-convergence}, we can apply
Reshetnyak's continuity theorem
\cite[Theorem 1.3]{Spe11} to $\lambda_i$ and $\lambda$. This gives
\begin{equation}\label{eq:reshetnyak-normal-convergence}
  \int_{\Sigma_{i,t_i}}f(x,\frac{d\lambda_i}{d|\lambda_i|})\,d\mu_{i,t_i}
  \to
  \int_{\partial^*\Omega^*}f(x,\frac{d\lambda}{d|\lambda|})\,d\mathcal H^n
\end{equation}
for every
$f\in C_c(\mathbb R^{n+1}\times{\mathbb{S}^n})$.
Since the Radon-Nikodym derivatives $d\lambda_i/d|\lambda_i|$, $ d\lambda/d|\lambda|$ are
unit normals of $\Sigma_{i,t_i}, \partial^*\Omega^*$ respectively, taking $f(x,\nu)=\varphi(x,\nu^\perp)$
gives convergence of the rectifiable varifolds
\[\Sigma_{i,t_i} \to \partial^*\Omega^*.\]

To identify the limiting mean curvature, fix
$X\in C_c^1(\mathbb R^{n+1};\mathbb R^{n+1})$ and set
\begin{equation}\label{fXtest}
    f_X(x,\nu)
  :=\operatorname{div}X(x)-\nu\cdot DX(x)\nu
    -h(0)X(x)\cdot\nu.
\end{equation}
Note that the second term eliminates the normal component of the ambient divergence.
Then, by definition of weak mean curvature,
\begin{align}\label{divxhxnu}
  \int_{\Sigma_{i,t_i}}f_X(x,\nu_i)\,d\mu_{i,t_i}
  &= \int_{\Sigma_{i,t_i}} \operatorname{div}_{\Sigma_{i,t_i}}X
  -h(0)X\cdot\nu_i\,d\mu_{i,t_i}\\
  &=\int_{\Sigma_{i,t_i}}
       (H_i-h(0))X\cdot\nu_i\,d\mu_{i,t_i}
\end{align}
The absolute value of the last integral is bounded by
\[
  \|X\|_\infty\left[
    A_i(t_i)^{1/2}
    \left(\int_{\Sigma_{i,t_i}}
       |H_i-h(t_i)|^2\,d\mu_{i,t_i}\right)^{1/2}
    +|h(t_i)-h(0)|A_i(t_i)
  \right],
\]
which tends to zero due to \eqref{Hi-hintegral2to0} and definition of $h(t)$ and $t_i\to 0$. Applying
\eqref{eq:reshetnyak-normal-convergence} to $f_X$ and the integral in \eqref{divxhxnu} converging to $0$,
we conclude that
\begin{equation}\label{weakmc}
\int_{\partial^*\Omega^*}
       \operatorname{div}_{\partial^*\Omega^*}X\,d\mathcal H^n
  =
  h(0)\int_{\partial^*\Omega^*}
       X\cdot\nu\,d\mathcal H^n.
\end{equation}
Thus, by comparing \eqref{weakmc} and the definition of weak mean curvature of finite perimeter set, we know that $\partial^*\Omega^*$ has constant weak mean curvature $h(0)$
with respect to the outward unit normal.
By Delgadino--Maggi's generalized Alexandrov's theorem \cite{DM19},
The boundary $\partial^*\Omega^*$ of $\Omega^*$ is finitely many union of spheres with equal radius.
Considering the equality case assumption, if
\[T_{\rm ext}(\Omega)= \frac{1}{2n}\left(\frac{P(\Omega^*)}{|{\mathbb{S}^n}|}\right)^{{\frac{2}{n}}},\]
$\Omega^*$ is only one  round ball. 

Now if $\partial \Omega$ has only one component and the equality holds, we have $\partial \Omega^*$ is a sphere. By the minimizing-hull obstacle property, the free part $\partial\Omega^* \setminus \partial\Omega$
is area minimizing and hence minimal (see 
\cite{HuiskenIlmanen2001} or
\cite[Theorem 5.3(ii)]{AFM22}). Since $\partial\Omega^*$ is a round
sphere and therefore has strictly positive mean curvature, the free
part must be empty. Hence the sphere 
$\partial\Omega^*$ is an open and closed  subset of $\partial\Omega.$
Since $\partial\Omega$ is connected, it follows that $\partial\Omega=\partial\Omega^*$,
and consequently $\Omega=\Omega^*$
is a round ball. This completes the proof of \eqref{eq:general-extinction6} and its equality case.

Finally, by the support inclusion fact that the support of Brakke flow is contained inside the level set flow \cite[Theorem 10.7]{Ilm94} and \eqref{eq:general-extinction6}, we obtain \eqref{eq:general-brakke-extinction6}. Now, we discuss the rigidity of \eqref{eq:general-brakke-extinction6} in the situation $\partial \Omega$ has only one component. Direct computation shows the multiplicity-one standard smoothly
shrinking round sphere Brakke flow achieves the equality in \eqref{eq:general-brakke-extinction6}. Conversely, if the equality in \eqref{eq:general-brakke-extinction6} holds,   by the above the support inclusion fact and equality case of   \eqref{eq:general-extinction6}, we have $\Omega=\Omega^*=B_R(x_0)$ for some $R>0$. Let  $M_t=\partial B_{r(t)}(x_0)$ with $r(t)=\sqrt{R^2-2nt}$.

Since the associated varifold $V_t$ of Brakke flow $\{\mu_t\}$ is integral and $\operatorname{spt}\mu_t\subset M_t$, there exists an
integer-valued measurable function $\theta_t$ such that
$V_t=\theta_t |M_t|$.
Since the generalized mean curvature is normal to the approximate tangent
plane, $H_t\perp TM_t$ for $\mu_t$-a.e.\ $x$. Hence, for every compactly
supported smooth vector field $X$ tangent to $M_t$,
\begin{equation}
     0
    =
    -\int H_t\cdot X\,d\mu_t
    =
    \delta V_t(X)
    =
    \int_{M_t}\theta_t\,\operatorname{div}_{M_t}X\,d\mathcal H^n.
\end{equation}
Thus $D_{M_t}\theta_t=0$ in the sense of distributions. Since $M_t$ is
connected, $\theta_t$ is constant $\mathcal H^n$-a.e.\ on $M_t$.
Consequently,
\begin{equation}
     V_t=m(t)|M_t|,
    \qquad\mu_t=m(t)\cH^n\llcorner M_t,\qquad
    m(t)\in\mathbb N\cup\{0\}.
\end{equation}
By Brakke's inequality in the definition of Brakke flow and the above fact $m(t)\in \mathbb{N}\cup\{0\}$, we have $m(t)\in \{0, 1\}$. If $m(t_0)=0$ for some $t_0<(2n)^{-1}P(\Omega^*)^{2/n}|{\mathbb{S}^n}|^{-2/n}$, then $T_{\textrm{ext}}(\{\mu_t\}, \Omega)<P(\Omega^*)^{2/n}|{\mathbb{S}^n}|^{-2/n}$, which contradicts with equality of \eqref{eq:general-brakke-extinction6}.  Hence, the integral Brakke flow $\{\mu_t\}$ is the multiplicity-one standard  smoothly shrinking round sphere Brakke flow.  This completes the  proof of Proposition \ref{thm:LSF-Brakke extinction optimal}.
\end{proof}

Now, we are ready to give a proof of Theorem \ref{Optimal extinction time estimates for weak mean curvature flow}.
\begin{proof}[Proof of Theorem \ref{Optimal extinction time estimates for weak mean curvature flow}]
By Proposition  \ref{thm:LSF-Brakke extinction optimal} and $P(\Omega^*)\leq P( \Omega)$, we have that \eqref{C1} and \eqref{C2} hold.

Then, the direct computation shows if $\Omega$ is round ball, equality in \eqref{C1} holds.  If $\Omega$ is a round ball and the Brakke flow $\{\mu_t\}$ is the multiplicity-one standard  smoothly shrinking round sphere Brakke flow, \eqref{C2} hold. On the other hand, if equalities in \eqref{C1}, \eqref{C2} hold, together with $P(\Omega^*)\leq P(\Omega)$,
\begin{equation}
    T_{\rm ext}(\Omega)\leq  \frac{1}{2n}\left(\frac{P(\Omega^*)}{|{\mathbb{S}^n}|}\right)^{{\frac{2}{n}}}\leq \frac{1}{2n}\left(\frac{P(\Omega)}{|{\mathbb{S}^n}|}\right)^{{\frac{2}{n}}}=  T_{\rm ext}(\Omega),
\end{equation}
we have 
\begin{equation}\label{P*=P}
    P(\Omega^*)=P(\Omega).
\end{equation}
Now we claim that $\partial\Omega$ has only one connected component. Suppose towards a contradiction if $\partial \Omega$ has more than one component, because $\partial\Omega$ is compact and smooth, it has finitely many
connected components:
\[
\partial\Omega
=
\Gamma_0 \sqcup \Gamma_1 \sqcup \cdots \sqcup \Gamma_{m},\quad m\geq 1
\]
and
\[
P(\Omega)
=
\mathcal{H}^n(\partial\Omega)
=
\sum_{j=0}^{m}\mathcal{H}^n(\Gamma_j).
\]
Let $C_\infty$ denote the unique unbounded connected component of $\mathbb{R}^{n+1}\setminus{\Omega}$.  Because the domain $\Omega$ itself is connected,
there is exactly one boundary component, say $\Gamma_0$, that faces
the unbounded complementary component $C_\infty$. Thus
$\partial C_\infty=\Gamma_0$ and
the remaining components $
\Gamma_1,\ldots,\Gamma_{m}$
surround bounded complementary components, i.e.\ holes of $\Omega$.

Now we let $\widehat{\Omega}
:=
\mathbb{R}^{n+1}\setminus{C_\infty}.$
Then $\widehat{\Omega}$ containing $\Omega$ is precisely $\Omega$ with all its bounded
complementary components filled in. Its boundary consists only of the
outer boundary $\partial\widehat{\Omega}=\Gamma_0$.
Consequently, by $m\geq 1$ we have
\[
P(\Omega^*)\leq P(\widehat{\Omega})
=
\mathcal{H}^n(\Gamma_0)<\mathcal{H}^n(\Gamma_0)
+
\sum_{j=1}^{m}\mathcal{H}^n(\Gamma_j)
=
P(\Omega).
\]
This contradicts \eqref{P*=P}, which implies $\partial \Omega$ is connected. By the above claim, the conclusions that  $\Omega$ is a round ball and the flow is standard multiplicity one smoothly shrinking round sphere follow from equality cases of Proposition  \ref{thm:LSF-Brakke extinction optimal}. This completes the proof of Theorem \ref{Optimal extinction time estimates for weak mean curvature flow}.
\end{proof}

Now, we give the proof of Theorem \ref{lpestimates arrival time} on the $L^p$ estimates of arrival time function $u$ of  smooth bounded mean convex domain.
\begin{proof}[Proof of Theorem \ref{lpestimates arrival time}]
Since the extinction time can be
interpreted as the $L^{\infty}$ norm of arrival time function, the optimal $L^\infty$ estimates of arrival time function $u$ in \eqref{eqlinfty} and its equality case directly follow from Theorem \ref{Optimal extinction time estimates for weak mean curvature flow}.

Then, for optimal $L^p$-estimates, we notice that for the arrival time function $u^*$ of $ \Omega^*$ as the outward minimizing hull of $\Omega$, $0 \leq u \leq u^*$ on $ \Omega$, so
\begin{equation}\label{uu*lp}
    \|u\|_{L^p(\Omega)}
\leq
\|u^*\|_{L^p(\Omega^*)}.
\end{equation}

Let $  K_t^*
:=
\{x\in\Omega^*:u^*(x)\geq t\}$, which is the level set flow starting from $\Omega^*$. Since $\Omega^*$ is the outward minimizing hull, $K_t^*$ is outward minimizing mean convex level set flow.   By the layer-cake formula and   sharp Euclidean isoperimetric inequality for set with finite Lebesgue measure  \cite[Theorem 14.1]{Mag12}, we have
\begin{align}\label{layercakeiso*}
     \int_{\Omega^*}(u^*)^p\,dx
&=
p\int_0^{T^{*}_{\textrm{ext}}} t^{p-1}|K_t^*|\,dt\leq \frac{p}{(n+1)|\mathbb{S}^n|^{\frac{1}{n}}}
\int^{T^{*}_{\textrm{ext}}}_{0} t^{p-1}P(K_t^*)^{\frac{n+1}{n}}dt,
\end{align}
where $T^{*}_{\textrm{ext}}$ is the extinction time of $K^*_t$ starting from $\Omega^*$. 

By Lemma \ref{lem:hull-approximation}, there exist a sequence of smooth strictly mean-convex domains $\Omega_i$ such that 
\begin{equation}\label{inclusionrelation omega}
    \Omega\subset\Omega^* \Subset \Omega_i
\end{equation}
and
\begin{equation}\label{Pomegetoomega*}
    P(\Omega^*)\leq P(\Omega_i)\leq P(\Omega^*)+i^{-1}
\end{equation}
Let $K_{i,t}$ be the mean-convex level-set flow starting from $\Omega_i$, and set
\begin{equation}
    A_i(t):=P(K_{i,t}).
\end{equation}
By integrating \eqref{eq:area-odi}, \eqref{eq:area-odi2} and extending $A_i(t)$ by zero beyond extinction,
we have
\begin{equation}\label{A_ibound1.4}
A_i(t)
\leq
\left[
P(\Omega_i)^{2/n}-c_i t
\right]_{+}^{n/2}
+\delta_i,
\qquad t\geq 0,
\end{equation}
where
\begin{equation}
\delta_i=i^{-1/2},
\qquad
c_i
=
2n|\mathbb{S}^n|^{2/n}
(1-\delta_i)^{2(n-1)/n}.
\end{equation}

By \eqref{inclusionrelation omega} and avoidance principle for level-set flow gives
\begin{equation}
K^*_t\subset K_{i,t}.
\end{equation}
Consequently, by the outward minimizing property of $K^*_t$, we have
\begin{equation}
P(K_t^*)\leq P(K_{i, t}^*)=A_i(t).
\end{equation}
Hence, by \eqref{A_ibound1.4}, we have
\begin{align}
    &\quad \frac{p}{(n+1)|\mathbb{S}^n|^{\frac{1}{n}}}
\int^{T^{*}_{\textrm{ext}}}_{0} t^{p-1}P(K_t^*)^{\frac{n+1}{n}}dt\\\nonumber
&\leq  \frac{p}{(n+1)|\mathbb{S}^n|^{\frac{1}{n}}}
\int^{T_i}_{0} t^{p-1}\left\{\left[
P(\Omega_i)^{2/n}-c_i t
\right]_{+}^{n/2}
+\delta_i \right\}^{\frac{n+1}{n}}dt
\end{align}
Then, by  the estimates \eqref{TTilimit} on the  extinction time $T_i$ of level set flow $K_{i, t}$ starting from $\Omega_i$ in the proof of Theorem \ref{Optimal extinction time estimates for weak mean curvature flow}  and \eqref{Pomegetoomega*}, integrating and sending $i\to \infty$, we obtain 
\begin{align}\label{limitingineq}
    &\quad \frac{p}{(n+1)|\mathbb{S}^n|^{\frac{1}{n}}}
\int^{T^{*}_{\textrm{ext}}}_{0} t^{p-1}P(K_t^*)^{\frac{n+1}{n}}dt\\\nonumber
&\leq
\frac{1}{(2n)^p}
\left[
\frac{1}{2}
\mathrm{B}\!\left(\frac{n+1}{2},\,p+1\right)
\right]
|\mathbb{S}^n|^{-\frac{2p+1}{n}}
P(\Omega^*)^{\frac{n+1+2p}{n}},
\end{align}
where $\mathrm{B}(\cdot, \cdot)$ 
is the standard {Beta} function defined in \eqref{eqn:beta}.  Then, by \eqref{uu*lp}, \eqref{layercakeiso*}, \eqref{limitingineq} and $P(\Omega^*)\leq P(\Omega)$ we have
\begin{align}\label{eqlp2}
\int_{\Omega}u^p dx
&\leq \int_{\Omega^*}(u^*)^p dx \leq \frac{p}{n+1}
|\mathbb{S}^n|^{-\frac{1}{n}}
\int^{T^{*}_{\textrm{ext}}}_{0} t^{p-1}P(K_t^*)^{\frac{n+1}{n}}dt, \\\notag
&\leq
\frac{1}{(2n)^p}
\left[
\frac{1}{2}
\mathrm{B}\!\left(\frac{n+1}{2},\,p+1\right)
\right]
|\mathbb{S}^n|^{-\frac{2p+1}{n}}
P(\Omega^*)^{\frac{n+1+2p}{n}}\\\notag 
&\leq
\frac{1}{(2n)^p}
\left[
\frac{1}{2}
\mathrm{B}\!\left(\frac{n+1}{2},\,p+1\right)
\right]
|\mathbb{S}^n|^{-\frac{2p+1}{n}}
P(\Omega)^{\frac{n+1+2p}{n}}.
\end{align}
%\beomjun{Here, after a linear change of variables in the time integral, we use
%the Beta function
%\begin{equation}\label{eqn:beta}
%\mathrm{B}(a,b):=\int_0^1 s^{a-1}(1-s)^{b-1}\,ds,
%\qquad a,b>0.
%\end{equation}}
Then, the optimal $L^p$-estimates of arrival time function $u$ in \eqref{eqlp} follows.

Now, if $\Omega$ is a ball, by direct computation  \eqref{eqlp} holds. If the equality case of \eqref{eqlp} holds, we have  $P(\Omega)=P(\Omega^*)$, so $\Omega$ is smooth outward minimizing and in particular  $\Omega$ is mean convex. Hence $K_t$ is also outward minimizing which is preserved under level set flow \cite[Theorem 3.5]{Whi00}. By the argument in the proof of Theorem \ref{Optimal extinction time estimates for weak mean curvature flow}, we also know  that $\Omega$ has only one smooth boundary component. 

Because the outward minimizing flow $K_t$ has smooth initial data $\Omega$ and is strictly mean convex flow at positive time by strong maximum principle, we can apply the Minkowski inequality for almost every time slice of outward minimizing level flow with smooth initial data in Proposition \ref{thm:ae-minkowski} with $\eta=0$. Then, let $A(t)=P(K_t)$, we have
 \be \frac{d}{dt}A(t) \leq 
    -n^2|\mathbb{S}^n|^{\frac2n}A(t)^{\frac{n-2}n} \qquad \text{for a.e. } t, \ee
so by the monotone property of $A(t)$, we can integrate and get
\begin{equation}
A(t) \leq
S(t)=\left[
P(\Omega)^{{\frac{2}{n}}}
- 2n |{\mathbb{S}^n}|^{{\frac{2}{n}}} t
\right]_+^{n/2}.
\end{equation}
Hence, we have
\begin{align}
\int_{\Omega}u^p dx&\leq \frac{p}{n+1}
|\mathbb{S}^n|^{-\frac{1}{n}}
\int^{T_{\textrm{ext}}(\Omega)}_{0} t^{p-1}A(t)^{\frac{n+1}{n}}dt, \\\notag
&\leq
 \frac{p}{n+1}
|\mathbb{S}^n|^{-\frac{1}{n}}
\int^{T_{*}}_{0} t^{p-1}S(t)^{\frac{n+1}{n}}dt, \\\notag
&=
\frac{1}{(2n)^p}
\left[
\frac{1}{2}
\mathrm{B}\!\left(\frac{n+1}{2},\,p+1\right)
\right]
|\mathbb{S}^n|^{-\frac{2p+1}{n}}
P(\Omega)^{\frac{n+1+2p}{n}}=\int_{\Omega}u^p dx,
\end{align}
where
\begin{equation}
    T_{\textrm{ext}}(\Omega)\leq T_* = \frac{1}{2n}
\left(\frac{P(\Omega)}{|{\mathbb{S}^n}|}\right)^{{\frac{2}{n}}}.
\end{equation}
Therefore, we have
\begin{equation}
 \int_0^{T_{\textrm{ext}}(\Omega)} t^{p-1}(S(t)^{\frac{n+1}{n}}-A(t)^{\frac{n+1}{n}})\,dt+\int_{T_{\textrm{ext}}(\Omega)}^{T_*} t^{p-1}S(t)^{\frac{n+1}{n}}\,dt=0,
\end{equation}
which implies
\begin{equation}
T_{\mathrm{ext}}(\Omega)
=T_*=
\frac{1}{2n}
\left(
\frac{P(\Omega^*)}{|{\mathbb{S}^n}|}
\right)^{\frac{2}{n}}
\end{equation}
holds. By the equality case of optimal extinction time estimates in Theorem \ref{Optimal extinction time estimates for weak mean curvature flow}, we know that $\Omega$ is a round ball and the flow is standard  smoothly shrinking round sphere.
This completes the proof of Theorem \ref{lpestimates arrival time}.
\end{proof}
Finally, we prove optimal  parabolic isoperimetric estimates for the space-time
track of outward minimizing mean curvature  flow in  Theorem \ref{sharp parabolic estimates}. 
\begin{proof}[Proof of Theorem \ref{sharp parabolic estimates}]
It suffices to prove for the case when $\partial\Omega$ is strictly mean convex outward minimizing.
This is because $\partial\Omega$ being outward minimizing and therefore mean convex implies that its positive time slices of the level set flow are strictly mean convex for a short time, as it coincides with smooth mean curvature flow.

For outward minimizing level set flow $K_t$ starting from $K_0$ with smooth hypersurface $\partial K_0=\partial \Omega$ as its boundary, it is  mean convex level set  flow with finite perimeter. Hence it is non-fattening level set flow by \cite[Corollary 3.3]{Whi00}, \cite{hershkovits2020nonfattening}. By the result of  Ilmanen in \cite[Theorem 11.4]{Ilm94} 
 (see also \cite[Theorem  2.3]{bernstein2018topology}), there is a unique  canonical matching motion Brakke flow $(T,\{\mu_t\}_{t\geq0})$, where $T$ is the unique undercurrent, $\mu_t$ is a unit density matching Brakke motion and in particular
\begin{equation}\label{matching motion condition}
\mu_{T_t}=\mu_t
=
\mathcal{H}^n\llcorner\partial^*K_t \quad\,\,\textrm{for}\,\, \mathrm{a.e}. t,
\end{equation}
where in our mean convex non-fattening level set flow situation, the sliced current $T_t$ of undercurrent $T$ satisfies
\begin{equation}
    T_t 
= \partial [K_t] \quad\,\,\textrm{for}\,\, \mathrm{a.e}. t.
\end{equation}
Then,
by the  property \eqref{matching motion condition} and 
 De Giorgi's characterization of perimeter $
P(K_t)
=
|D\chi_{K_t}|(\mathbb R^{n+1})
=
\mathcal H^n(\partial^*K_t)$ in \cite[Theorem 15.9]{Mag12}, 
we have
\begin{align}\label{mass area perimeter relation}
\mathbf{M}(T_t)
:=\mu_{T_t}(\mathbb{R}^{n+1})=
\mu_t(\mathbb{R}^{n+1})
=
P(K_t) \quad\,\,\textrm{for}\,\, \mathrm{a.e}. t.
\end{align}
Then, by the parabolic co-area formula in \cite[Theorem D]{hershkovits2017isoperimetric}, \cite[Example 3.5]{hershkovits2017isoperimetric} and \eqref{mass area perimeter relation}, we have
\begin{align}\label{eq:general-parabolic}
\mathcal{H}_{\textrm{par}}^{n+2}(X)
= \frac{\pi}{2(n+2)} \int_0^\infty \mathbf{M}[T_t]\,dt= \frac{\pi}{2(n+2)} \int_0^\infty P(K_t)\,dt,
\end{align}
where  $X\subset\R^{1, n+1}$ is the
rectifiable set associated to current $T$.

Now, we let $A(t)=P(K_t)$. We note that $\Omega$ is mean convex, since it is outward minimizing. Putting suitable cutoff function in Lemma \ref{lem:no-mass-drop} for mean convex level set flow, we obtain
\begin{equation}
   A'(t)   =
        -\int_{\partial K_t} |\vec H|^2\,d\mathcal{H}^n\leq -A(t)^{-1}\left(\int_{\partial K_t}H\right)^2
        \qquad \text{for a.e. } t.
\end{equation}
Then, by the Minkowski inequality for outward minimizing level flow with smooth initial data \footnote{Here we do not need the approximation argument because in  Theorem \ref{sharp parabolic estimates}, Proposition  \ref{thm:LSF-Brakke extinction optimal}, we assumed smoothness of outward minimizing initial data and the proof of Lemma \ref{thm:ae-minkowski}  is based  on elliptic regularization argument.} in Lemma \ref{thm:ae-minkowski} with $\eta=0$ as well as the fact that  outward minimizing level flow is strictly mean convex at positive time by strong maximum principle, we have
 \be \frac{d}{dt}A(t) \leq 
    -n^2|\mathbb{S}^n|^{\frac2n}A(t)^{\frac{n-2}n} \qquad \text{for a.e. } t. \ee
Hence,
\begin{equation}
A(t) \leq
\left[
A(0)^{{\frac{2}{n}}}
- 2n |{\mathbb{S}^n}|^{{\frac{2}{n}}} t
\right]^{n/2}.
\end{equation}
and for outward minimizing flow we again obtain 
\be T_{\textrm{ext}} \leq \frac{1}{2n}\left(\frac{A(0)}{|\mathbb{S}^n|}\right)^{{\frac{2}{n}}}\ee
as in Proposition  \ref{thm:LSF-Brakke extinction optimal}.

Combining the above two estimates, we obtain
\be\label{intatestimates}
\int_0^{T_{\textrm{ext}}} A(t)\,dt\leq \int_0^{T_{\textrm{ext}}}\left[
A(0)^{{\frac{2}{n}}}
- 2n |{\mathbb{S}^n}|^{{\frac{2}{n}}} t
\right]^{n/2} dt \leq
\frac{A(0)^{1+2/n}}
{n(n+2)|{\mathbb{S}^n}|^{{\frac{2}{n}}}},
\ee
and
\begin{align}
\mathcal{H}_{\textrm{par}}^{n+2}(X)
= \frac{\pi}{2(n+2)} \int_0^\infty P(K_t)\,dt\leq \frac{\pi P(\Omega)^{\frac{n+2}{n}}}
{2n(n+2)^2|{\mathbb{S}^n}|^{{\frac{2}{n}}}},
\end{align}
which verifies the estimates \eqref{eq:sharp-parabolic-matching}. 

Then, we discuss the equality case. By direct computation, it is easy to see the spherical level set flow achieves the equality in \eqref{eq:sharp-parabolic-matching}. Now if the equality in \eqref{eq:sharp-parabolic-matching} holds, we have 
\begin{align}
\int_0^{T_{\textrm{ext}}}A(t)\,dt=
\frac{P(\Omega)^{\frac{n+2}{n}}}
{n(n+2)|{\mathbb{S}^n}|^{{\frac{2}{n}}}}=
\int_0^{T_*}
\left[
P(\Omega)^{{\frac{2}{n}}}
-2n|{\mathbb{S}^n}|^{{\frac{2}{n}}}t
\right]^{n/2}dt,
\end{align}
where
\begin{equation}
    T_* = \frac{1}{2n}
\left(\frac{P(\Omega)}{|{\mathbb{S}^n}|}\right)^{{\frac{2}{n}}},
\end{equation}
by $A(0)=P(\Omega)$ and \eqref{intatestimates},  we have
\begin{equation}
  \int_0^{T_{\textrm{ext}}}
\left[
P(\Omega)^{{\frac{2}{n}}}
-2n|{\mathbb{S}^n}|^{{\frac{2}{n}}}t
\right]^{n/2}dt  =  \int_0^{T_*}
\left[
P(\Omega)^{{\frac{2}{n}}}
-2n|{\mathbb{S}^n}|^{{\frac{2}{n}}}t
\right]^{n/2}dt.
\end{equation}
In particular, by our optimal extinction time estimates $T_{\textrm{ext}}(\Omega)\leq T_*$ 
and $P(\Omega)^{{\frac{2}{n}}}
-2n|{\mathbb{S}^n}|^{{\frac{2}{n}}}t>0$ for $t\in [0, T_*)$, we have
\begin{equation}
    T_{\textrm{ext}}(\Omega)=T_*= \frac{1}{2n}
\left(\frac{P(\Omega)}{|{\mathbb{S}^n}|}\right)^{{\frac{2}{n}}}.
\end{equation}
Then, by the rigidity case in Proposition  \ref{thm:LSF-Brakke extinction optimal} for outward minimizing level set flow, we conclude that
$\Omega$ is a round ball and the flow is standard  smoothly shrinking round sphere. This completes the  proof of Theorem \ref{sharp parabolic estimates}.
\end{proof}

\subsection{Some remarks and problems}
In the end of the paper, we make some remarks and propose some problems about the related topics to our results.
\begin{remark}\label{remarkonsmooth issue}
    The assumption that $\partial\Omega$ is a smooth hypersurface in our  Theorem \ref{Optimal extinction time estimates for weak mean curvature flow}, Theorem \ref{lpestimates arrival time}, Theorem \ref{sharp parabolic estimates} and Proposition  \ref{thm:LSF-Brakke extinction optimal} is used in the elliptic regularization argument of Proposition \ref{thm:ae-minkowski}  and  invoking Lemma \ref{lem:hull-approximation}
    for obtaining the smooth, but not necessarily mean convex exterior approximations $U_i$.
    Following \cite{Schmidt2015}, the approximation step only requires
\begin{equation}\label{eq:general-assumption}
    \cH^n(\partial\Omega^*)=P(\Omega^*),
    \end{equation}
    which is satisfied when $\partial\Omega$ is smooth. However,  the inequality in Theorem \ref{Optimal extinction time estimates for weak mean curvature flow} indeed holds for bounded Lipschitz domain due to the smoothing effect \cite{ecker1989mean, ecker1991interior}, and the rigidity statement holds by a slight modification of Lemma \ref{lem:hull-approximation}, Proposition \ref{thm:LSF-Brakke extinction optimal}, and Theorem \ref{Optimal extinction time estimates for weak mean curvature flow}\footnote{The equality in Theorem \ref{Optimal extinction time estimates for weak mean curvature flow} implies that $\Omega$ is outward minimizing. Since $\partial\Omega$ is Lipschitz, \eqref{eq:general-assumption} holds with $\Omega^*$ replaced by $\Omega$. Then the statement and proof of Lemma \ref{lem:hull-approximation} are valid verbatim with $\Omega$ in place of $\Omega^*$. Following the similar argument in the proof of  Proposition \ref{thm:LSF-Brakke extinction optimal} and Theorem \ref{Optimal extinction time estimates for weak mean curvature flow}, we obtain a subsequential limit $F\supset\Omega$ of $\{\Omega_i\}$ such that $P(F)=P(\Omega)$, and we prove that $F$ is a round ball. Because $\Omega$ is outward minimizing and the boundary of $F$ has positive curvature at every point, we conclude $\Omega$ is identical to the round ball $F$. Then the equality case in Theorem \ref{Optimal extinction time estimates for weak mean curvature flow} follows.}. Theorem \ref{lpestimates arrival time} is expected to hold for $C^1$ domain with  nonnegative boundary weak mean curvature in $L^2$, and  Theorem \ref{sharp parabolic estimates} is expected to hold for $C^1$ outward minimizing domain with  nonnegative boundary weak mean curvature in $L^2$.
    %Therefore, the smoothness assumption of Theorem \ref{Optimal extinction time estimates for weak mean curvature flow} and Proposition \ref{thm:LSF-Brakke extinction optimal} can be replaced with somewhat general \eqref{eq:general-assumption}.
\end{remark}
\begin{remark}
    For general bounded domain with smooth boundary $\Omega$, taking outward minimizing hull will reduce the perimeter, and we can only guarantee
    the existence of an enhanced motion \cite[Section 8]{Ilm94},
    where the slice measure of the undercurrent need not be a Brakke flow, and the multiplicity in the enhanced motion is also not necessary to be $1$.
    In this case, it is generally challenging to track the evolution of slice mass $\mathbf{M}(T_t)$ for obtaining optimal parabolic isoperimetric estimates as in Theorem \ref{sharp parabolic estimates}.
  On the other hand,  with our optimal extinction time estimate (Proposition  \ref{thm:LSF-Brakke extinction optimal}), we can still follow the proof of Hershkovits \cite[Theorem C]{hershkovits2017isoperimetric} and adapt the varifold version Michael-Simon Sobolev inequality in \cite{tsiamis2026sharp, santilli2026optimal} to replace the Minkowski inequality in Section \ref{intro}, it is expected that one can obtain improved but nonsharp estimate for the parabolic measure of the enhanced
motion:
\begin{align}\label{eq:improved general-parabolic}
\mu_{\mathrm{par}}(X)
    &\!\leq\! \frac{\pi}
{2n(n+2)^2 |\mathbb{B}^n|^{\frac{2}{n}}}
\!\left[\!
P(\Omega)^{\frac{n+2}{n}}\!
-\!
\left(\!
P(\Omega)^{{\frac{2}{n}}}
\!\!-\!\!
\left(\!
\frac{|\mathbb{B}^n|}{|\mathbb{S}^n|}
P(\Omega^*)
\!\right)^{{\frac{2}{n}}}
\!\right)^{\frac{n+2}{2}}\!
\right],
\end{align}
where $\mu_{\mathrm{par}}=\theta \mathcal{H}^{n+2}_{\textrm{par}}\llcorner X$, $\theta$ is the multiplicity in the enhanced motion and $\Omega^*$ is the outward minimizing hull of $\Omega$.  
%On the other hand, because the crucial ingredient is Minkowski inequality\footnote{Chodosh, Eichmair and Koerber \cite{chodosh2026minkowski} showed that Minkowski inequality generally fails for small perturbations of a round
%sphere.}, it is expected that the optimal parabolic isoperimetric estimates also holds when $\Omega$ is mean convex and star-shaped due to \cite{GuanLi2009}.
\end{remark}
\begin{remark}
It would also be interesting to explore  the optimal extinction time estimates and parabolic isoperimetric estimates of mean curvature flow in general ambient manifolds.  For mean curvature flow of smooth
two-dimensional surfaces in Cartan–Hadamard
manifolds of arbitrary dimension, we expect that one can obtain optimal  extinction time estimates and optimal parabolic isoperimetric estimates. More precisely, using the argument in  Hershkovits \cite[Theorem C]{hershkovits2017isoperimetric} and the sharp Willmore inequality  \cite[Theorem 1.4]{schulze2020optimal}
\begin{equation}
\int_{\Gamma} |H|^2\,d\mu
\geq
16\pi+4\kappa |\Gamma|
=
2^2|\mathbb S^2|+4\kappa |\Gamma|,\quad -\kappa\leq 0
\end{equation}
for any integral
$2$-varifold $\Gamma$ in $m$-dimensional Cartan-Hadamard manifold $(N^{n+2}, g)$ with sectional curvature bounded above by $-\kappa\leq 0$,   it is  expected that   for  integral enhanced motion of Brakke flow starting from $2$-dimensional smooth surface $\Sigma$, one can obtain the following optimal extinction time estimates 
\begin{equation}\label{eq:sharp-extinction-CH}
T_{\mathrm{ext}}
\le
\begin{cases}
\displaystyle
\frac{1}{4\kappa}
\log\left(
1+\frac{\kappa |\Sigma|}{4\pi}
\right),
& \kappa>0,\\[1.2ex]
\displaystyle
\frac{|\Sigma|}{16\pi},
& \kappa=0,
\end{cases}
\end{equation}
and optimal parabolic isoperimetric estimates 
\begin{equation}\label{eq:sharp-parabolic-measure-CH}
\mu_{\mathrm{par}}(X)
\le
\begin{cases}
\displaystyle
\frac{\pi |\Sigma|}{32\kappa}
-
\frac{\pi^2}{8\kappa^2}
\log\left(
1+\frac{\kappa |\Sigma|}{4\pi}
\right),
& \kappa>0,\\[1.2ex]
\displaystyle
\frac{|\Sigma|^2}{256},
& \kappa=0,
\end{cases}
\end{equation}
and the above equalities are achieved if and only if $\Sigma$ is  a smooth embedded two dimensional sphere,  bounds a totally geodesic  embedded three dimensional ball   that is isometric to a geodesic ball in the three dimensional simply connected space form of constant sectional curvature $-\kappa\leq 0$. 
\end{remark}
\begin{remark}\label{problems}
    Finally, we propose the following related problems.
    \begin{enumerate} 
        \item Does Theorem \ref{sharp parabolic estimates} hold for more general domains with smooth boundary such as mean convex smooth domains? 
        \item Do  Theorem \ref{Optimal extinction time estimates for weak mean curvature flow}, Theorem \ref{lpestimates arrival time}, and Theorem \ref{sharp parabolic estimates} hold for the situation of general dimensional Cartan-Hadamard manifolds and of general codimensional case?
        \item Are there quantitative stability versions  of Theorem \ref{Optimal extinction time estimates for weak mean curvature flow}, Theorem \ref{lpestimates arrival time}, and Theorem \ref{sharp parabolic estimates}?
         \item Let $\omega^k_1(\partial\Omega)$ be the first $k$-dimensional min-max width of smooth closed hypersurface $\partial \Omega$ and $1\leq k\leq n$ (see  \cite[page 942, (7)]{liokumovich2018weyl}, \cite{guth2009minimax}, \cite{gromov2006dimension} for the definition, or one can also consider other types of min-max width such as Almgren-Pitts min-max width \cite{chodosh2023p}). Is the following extinction time estimates 
        \begin{equation*}
        T_{\mathrm{ext}}(\Omega)\leq \frac{1}{2k}\left(\frac{\omega^k_1(\partial\Omega)}{|\mathbb{S}^k|}\right)^{\frac{2}{k}}
         \end{equation*}
         valid and optimal for certain class of mean curvature flows (for example, convex closed mean curvature flows)?  Our Theorem \ref{Optimal extinction time estimates for weak mean curvature flow} affirmatively answers  this problem when $k=n$. Colding-Minicozzi \cite{colding2008width} affirmatively answered a related version question for convex closed mean curvature flow  when $k=1, 2$, where the definitions of widths $W_1(0), W_2(0)$ in \cite{colding2008width}  are somewhat different from  the definition of widths in \cite[page 942, (7)]{liokumovich2018weyl}.
    \end{enumerate}
\end{remark}

\bibliographystyle{alpha}
\bibliography{reference}

\vspace{3mm}

{\sc \noindent Beomjun Choi, Department of Mathematical Sciences, Korea Advanced Institute of Science
and Technology, 34141 Daejeon,  Republic of
Korea} 

{E-mail:} bchoi@kaist.ac.kr

{\sc \noindent Wenkui Du, School of Mathematics, Hunan University, Changsha, 410082, Hunan Province, China}

{E-mail:} duwenkui@hnu.edu.cn

{\sc \noindent Seung Chul Park, Department of Mathematical Sciences, Korea Advanced Institute of Science
and Technology, 34141 Daejeon,  Republic of
Korea}

{E-mail:} sisyphe@kaist.ac.kr

{\sc \noindent Junseo Youn, Mathematics Department, Pohang University of Science and Technology, 37673 Pohang, Republic of Korea }

{E-mail:} youncurtis@postech.ac.kr.

\end{document}